\documentclass{amsart}
\usepackage{amssymb}
\usepackage{wasysym}
\usepackage{amsfonts}
\usepackage{bm}
\usepackage[dvips]{graphics}
\usepackage{graphicx}
\usepackage{epsfig}
\usepackage{float}
\usepackage{enumitem}

\newtheorem{thm}{Theorem}[section]
\newtheorem{prop}[thm]{Proposition}
\newtheorem{cor}[thm]{Corollary}
\newtheorem{lemma}[thm]{Lemma}

\theoremstyle{definition}
\newtheorem{definition}[thm]{Definition}

\theoremstyle{definition}
\newtheorem{exam}[thm]{Example}
\newtheorem{rem}[thm]{Remark}

\theoremstyle{definition}
\newtheorem*{prooff}{\it{Proof of Theorem 4.4}}

\newcommand{\vs}{\mathcal{V}}

\newcommand{\orbit}{\mathcal{O}}

\newcommand{\reals}{\mathbb{R}}

\newcommand{\sections}{\Gamma(E)}

\newcommand{\reg}{\rho^{\text{reg}}_G}
\newcommand{\regd}{\rho^{\text{reg}}_{D_n}}

\numberwithin{equation}{section}

\begin{document}


\vspace{1cm}

\title[Factorization of the Stability Polynomials of Ring Systems]{Factorization of the Stability Polynomials of Ring Systems}
\author{Eduardo S. G. Leandro\\}
\address{Universidade Federal de Pernambuco\\
	  Depto de Matem\'atica \\
	  Av. Jornalista An\'{\i}bal Fernandes s/n, Recife, PE \\
	  50740-560, Brasil\\
	  Phone: 55 81 2126-7650, Fax: 55 81 2126-8410
	  }
\email{eduardo@dmat.ufpe.br}
\thanks{I wish to acknowledge the support of Cristina Stoica and Manuele Santoprete, as well as their participation in the early discussions of the research project which led to this article. I also wish to acknowledge the hospitality of the Department of Mathematics at Wilfrid Laurier University, where this research was initiated.}

\begin{abstract}
 Let $D_n$ be the dihedral group with $2n$ elements, and suppose $n$ is greater than one. We call ring system a finite $D_n$-symmetric set of points in $\reals^2$. Ring systems have been used as models for planets surrounded by rings, and may be seen as relative equilibria of the $N$-body or the $N$-vortex problem. As a first significant step towards linear stability analysis, we study the factorization of the stability polynomial of an arbitrary ring system by systematically exploiting the ring's symmetry through representation theory of finite groups. Our results generalize contributions by J. C. Maxwell from mid-XIX century until contemporary authors such as J. Palmore and R. Moeckel, among others. 
\end{abstract}
\maketitle

\section{Introduction} \label{sec0}


 In his celebrated, award-winning essay published in 1859, the mathematical physicist J. C. Maxwell analyzed several models for the gravitational system formed by Saturn and its rings. Maxwell concluded that the only linearly stable model consisted of a solution of the planar $N$-body problem with a sufficiently massive body at the center of a uniformly rotating regular $n$-gon whose vertices are occupied by bodies of small equal masses. 
A crucial step in Maxwell's study was a linear change of coordinates made in the vector space formed by the planar displacements of the $n$-gon. Viewed in terms of complex variables, Maxwell's change of coordinates amounts essentially to a discrete Fourier transform, see for instance the lectures by Poincar\'e~\cite{p03}, and remarks in~\cite{sv91}. As a consequence of such a change of coordinates, the characteristic polynomial associated with the linearization of the equations of motion, also known as the secular or stability polynomial, is factored into polynomials of degrees two and four in a variable $\lambda^2$. Notice the degrees of the factors do not depend on $N=n+1$, the total number of bodies. Analogous factorizations were obtained by Palmore~\cite{p76}, in his study of degenerate relative equilibria in the Newtonian $N$-body problem, and later on, by Moeckel~\cite{m95}, who used the change of basis in~\cite{p76} to decompose the spaces of planar displacements of the regular $n$-gon and of the centered regular $n$-gon into invariant subspaces under the stability matrix. It seems evident that the factorizations obtained by Maxwell, Palmore and Moeckel are due to the symmetry of the centered regular $n$-gon, however a systematic discussion of this relationship seems yet absent from the literature. One of the chief purposes of the present article is to contribute to such systematization. We notice that the question of linear stability of relative equilibria with respect to normal displacements has been answered in the affirmative in~\cite{m95}.
 
  A finite symmetric subset of $\reals^2$ can only have a cyclic group or dihedral group as its symmetry group~\cite{a91}. The symmetry group of the set formed by the vertices of a regular $n$-gon and the point at its barycenter is the dihedral group $D_n$. One may ask the reciprocal question: what is the form of a general subset $X$ of $\reals^2$ possessing dihedral symmetry? It turns out a complete and relatively simple answer can be given for all $n \geq 2$: $X$ consists of $a=0$ or 1 point at a position $O$ surrounded by $b$ regular $n$-gons and $c$ semiregular $2n$-gons having $O$ as their common center.  Due to the geometric appearance of this structure, we refer to $X$ as a ring system of type $(a,b,c)$. The dihedral symmetry of $X$ allows us to define a canonical representation $\sigma^E$ (i.e., an action through isomorphisms) of the group $D_n$ on the space of planar displacements of $X$. We verify that $\sigma^E$ can be expressed as the direct sum of copies of a familiar representation, namely the regular representation of $D_n$, one copy for each regular $n$-gon and two copies for each semiregular $2n$-gon in $X$. If $X$ contains its barycenter $O$, an additional copy of the standard representation $\sigma$ of $D_n$ in the plane must be added to form $\sigma^E$.
  
 We consider a ring system which uniformly rotates about its barycenter as a relative equilibrium of the $N$-body problem under an arbitrary homogeneous potential, or of the  $N$-vortex problem. In the particular context of the Newtonian $N$-body problem, relative equilibria are the only known explicit periodic solutions and have been the object of research since Euler and Lagrange. It is known that some ring systems consisting of concentric regular $n$-gons with or without a body at the barycenter are relative equilibria provided the masses at the vertices of individual regular $n$-gons are equal and the radii of the corresponding circumscribing circles are suitably chosen. However, the question of whether a general ring system is a relative equilibrium for some choice of masses and radii appears to be open. Since our main concern are the possible factorizations of stability polynomials, we will treat all ring systems as relative equilibria under the assumption of equal masses (vorticities) at the vertices of each individual regular $n$-gon and each individual semiregular $2n$-gon.
 
 Group representation theory provides standard tools to exploit symmetry in the analysis of physical models~\cite{fs94,mw93,s94}. In the present context, the symmetry of a ring system $X$ implies that the hessian of the potential function commutes with the matrices associated with the canonical representation $\sigma^E$; we refer to this fact by saying that the hessian is $\sigma^E$-equivariant, or simply equivariant. Representation theory provides projection operators associated with the irreducible representations of the group $D_n$, and we can apply such operators to decompose the space of planar displacements into subspaces invariant by the hessian. The dimensions of these invariant subspaces depend only on the type $(a,b,c)$ of $X$. The matrix in the linearized equations of motion at the configuration corresponding to $X$ can be expressed as the sum of an equivariant matrix $A$ (the hessian multiplied by another equivariant matrix, namely the inverse of the mass matrix), the identity matrix $I$ and an antisymmetric matrix $J$. It turns out that $J$ interchanges the $A$-invariant subspaces produced by representation theory. This effect is precisely the one noticed in~\cite{m95} for ring systems of types $(0,1,0)$ and $(1,1,0)$. For a general ring system, using our understanding of the relationship between $J$ and the $A$-invariant subspaces, one can form subspaces simultaneously invariant by $A$ and $J$, and thus block diagonalize any linear combinations of $A$, $I$ and $J$. 

 We state our main conclusions and some other relevant contents of the article. The stability polynomial of a ring system of type $(a,b,c)$ and symmetry group $D_n$, $n>2$, factors as one ($n$ odd) or two ($n$ even)  polynomials of degree $2(b+2c)$, $n/2-3/2$ ($n$ odd) or $n/2-2$ ($n$ even) polynomials of degree $4(b+2c)$, and one polynomial of degree $2(a+2b+4c)$. The latter polynomial possesses a degree two factor due to the translational symmetry of the potential functions of the $N$-body problems; moreover, the rotational symmetry and homogeneity of the potential functions produce a factor of degree two of one of the polynomials of degree $2(b+2c)$. The (non-unique) bases of the space of planar displacements of $X$ which lead to the above factorizations are explicitly constructed in a coordinate-free manner by making use of the definitions and operational properties of the projection and other useful operators from representation theory, as well as $J$. The displacements in each such basis can  be illustrated through diagrams, and some of them are depicted in the figures of the latter sections of the paper. The suitable bases can be chosen orthogonal with respect to the inner product defined by the mass parameters and, under a normalization condition, such bases are also symplectic with respect to a symplectic form defined by $J$ and the masses. In addition, these bases can be modified in order to include the displacements associated with the symmetries and homogeneity of the potential functions; such displacements are eigenvectors of the matrix $A$ of the previous paragraph, and are pairwise transposed by $J$, thus leading to the latter factorizations mentioned above.  
Finally, we identified a few cases where full factorization is achievable, namely when $N=4$ and $n=2$. Among the associated relative equilibria are the rhombuses formed by two pairs of equal masses. All these results hold for the $N$-vortex problem with masses replaced by vorticities.

   Besides the conclusions listed above, our approach circumvents, in the symmetric setting, some lengthy matrix calculations to which one is led when investigating the linear stability of systems with an arbitrary number of particles. On the other hand, although we provide explicit relations between the entries of the block diagonal matrices leading to the factorization of the stability polynomials, important issues such as determination of the factors and the examination of the nature of their respective roots are left for subsequent works.

\section{Relative Equilibria of the Planar $N$-Body Problem and $N$-Vortex Problem} \label{sec1}

 In this section we deduce the expression for the stability polynomial of a relative equilibrium of the $N$-body and of the $N$-vortex problems. For the purpose of defining relative equilibria, it suffices to set up both problems in the Euclidean plane.

 
 The planar $N$-body problem consists of studying the dynamics of $N$ particles interacting according to a homogeneous force law:
\[ \ddot{q}_i=\sum_{j \neq i} m_j ||q_i-q_j||^{2\gamma}(q_j-q_i),\ \ q_i \in \reals^2,\]
$m_i \in \reals \setminus \{0\}$, $i=1,\cdots,N$, and $\gamma \in \reals \setminus \{-1\}$. When $\gamma=-\frac{3}{2}$ and $m_i>0$ for all $i$, we have the planar Newtonian $N$-body problem and the $m_i$ are called masses. If $\gamma=-1$, we have a system of first-order differential equations
\[ \dot{q}_i=-\mathcal{K}\sum_{j \neq i} m_j ||q_i-q_j||^{-2}(q_j-q_i),\]
with $q_i,m_i$ as above, and 
\begin{equation}\label{rotperp}
\mathcal{K}=\left[\begin{array}{cc}
																								0&-1\\
																								1&0
																								\end{array}
																					\right].
\end{equation}
This is the Helmholtz $N$-vortex problem, and the parameters $m_i$ are known as vorticities or circulations. In the sequel we formulate the $N$-body problem and the $N$-vortex problem as Hamiltonian systems, define relative equilibria and determine their stability polynomials.
																					
 Firstly, for each $\gamma\in \reals \setminus \{ -1\}$, we define the potential energy by
\[ U_{\gamma}(q)=\frac{1}{2\gamma+2} \sum_{i <j} m_i m_j ||q_i-q_j||^{2\gamma+2}.\]
Let $p_1,\cdots,p_N \in (\reals^2)^* \simeq \reals^2$ be the respective conjugate momenta of $q_1,\cdots,q_N \in \reals^2$. The Hamiltonian of the planar $N$-body problem is
\[ H_{\gamma}(q,p)=\frac{1}{2} \sum_{i=1}^N \frac{||p_i||^2}{m_i}+U_{\gamma}(q), \ \ q=(q_1,\cdots,q_N),\ \ p=(p_1,\cdots,p_N).\]
We endow the phase space $\reals^{2N} \times \reals^{2N}$ with the standard symplectic structure $\mathbb{J}=\left[\begin{array}{cc}
																								0&I_{2N}\\
																								-I_{2N}&0
																								\end{array}
																					\right],$ where $I_k$ from now on represents the identity matrix of size $k\times k$. If $z=(q,p)$, the planar $N$-body problem can be written as
\begin{equation*} \label{nbpa}
 \dot{z}=\mathbb{J}\nabla H_{\gamma}(z).
\end{equation*}

 Next let us consider the $N$-vortex problem. The Hamiltonian is 
\[ H_{-1}(q)=-\sum_{i < j} m_i m_j \ln ||q_i-q_j||, \qquad q \in \reals^{2N}.\]
We endow the phase space $\reals^{2N}$ with the symplectic structure $\mathbb{K}=J M^{-1}$, where $J$ and $M$ are the block-diagonal matrices:
\begin{equation} \label{jem}
 J=\left[\begin{array}{ccc}
					 \mathcal{K} & & \\
					 & \ddots &\\
					 & & \mathcal{K}
					 \end{array}
			\right] \quad \text{and} \quad M=\left[\begin{array}{ccc}
					 m_1 I_2 & & \\
					 & \ddots &\\
					 & & m_N I_2
					 \end{array}
			\right],
\end{equation}
with $\mathcal{K}$ given in~\eqref{rotperp}. The $N$-vortex problem in Hamiltonian form is
\begin{equation*} \label{nvp}
 \dot{q}= \mathbb{K}\nabla H_{-1} (q).
\end{equation*}

 In order to define relative equilibria, we change to a reference frame in uniform rotation around the origin. Let $\omega \in \reals \setminus \{0\}$ and let $R(\omega t)$ denote the rotation around the origin by an angle $\omega t$. The new coordinates are given by
\begin{equation} \label{rotcoords}
 x_i=R(\omega t)q_i, \ \ y_i=R(\omega t)p_i, \ \ i=1,\cdots, N.
\end{equation}
It is clear that this is a (time-dependent) symplectic change of coordinates $z \mapsto \zeta(t,z)$ in $\reals^{2N}\times \reals^{2N}$. The new Hamiltonian is $\mathcal{H}_{\gamma}(\zeta)=H_{\gamma}(z(\zeta))+\mathcal{R}(t,\zeta)$, where, for $\gamma=-1$, the remainder $\mathcal{R}$ satisfies the equation:
\[ \frac{\partial \zeta}{\partial t}=\mathbb{J} \nabla \mathcal{R}(t,\zeta).\]
From equations~\eqref{rotcoords}, and the fact that $R'(\omega t)R(-\omega t)=\mathcal{K}$, it follows that
\[ \mathcal{R}(t,\zeta)=\omega\sum_{i=1}^N x_i \mathcal{K} y_i=\omega \varkappa^T J \nu,\]
where we view $\varkappa=(x_1,\cdots,x_N)$ and $\nu=(y_1,\cdots, y_N)$ as column matrices and $J$ is as in~\eqref{jem}. Thus the Hamiltonian of the planar $N$-body problem in rotating coordinates $\zeta=(\varkappa,\nu)$ is
\begin{equation*} \label{hamilrot}
\mathcal{H}_{\gamma}(\varkappa,\nu)=\frac{1}{2}\nu^T M^{-1} \nu+\omega \varkappa^TJ\nu+U_{\gamma}(\varkappa), \ \ \gamma \neq -1,
\end{equation*}
with $M$ as in \eqref{jem}. By a similar procedure, now in $\reals^{2N}$, we obtain the Hamiltonian for the $N$-vortex problem in rotating coordinates:
\begin{equation*}\label{hamilrotv}
\mathcal{H}_{-1}(\varkappa)=\frac{\omega}{2} \varkappa^T M \varkappa+H_{-1}(\varkappa).
\end{equation*}
We are now ready to define relative equilibrium.

\begin{definition} \label{releq} A \emph{relative equilibrium} of the planar $N$-body (resp. $N$-vortex) problem is an equilibrium solution of the Hamiltonian system defined by $\mathcal{H}_{\gamma}$ (resp. $\mathcal{H}_{-1}$).
\end{definition}
It follows that a relative equilibrium of the $N$-body (resp. $N$-vortex) problem is a solution of $\nabla \mathcal{H}_{\gamma}=0$, $\gamma \neq -1$ (resp. $\nabla \mathcal{H}_{-1}=0)$. Thus a relative equilibrium must satisfy the equations:
\begin{equation*} \label{releq}
\omega^2 M \varkappa=-\nabla U_{\gamma}(\varkappa), \ \  \text{for} \ \gamma \neq -1, \quad \text{or} \quad \omega M \varkappa=-\nabla H_{-1}(\varkappa).
\end{equation*}

\subsection{Linearization and the Stability Polynomials} 

 We linearize the Hamiltonian systems of differential equations
\[ \dot{\zeta}=\mathbb{J}\nabla \mathcal{H}_{\gamma}(\zeta),\ \ \gamma \neq -1,\quad \text{and} \quad \dot{\varkappa}=\mathbb{K}\nabla \mathcal{H}_{-1}(\varkappa) \]
at a relative equilibrium. In the case $\gamma \neq -1$, we compute the $4N \times 4N$ matrix
\[ L_{\gamma}(\varkappa)=D(\mathbb{J}\nabla \mathcal{H}_{\gamma})(\varkappa, \nu)=\left[\begin{array}{cc}
																						-\omega J & M^{-1}\\
																						-D\nabla U_{\gamma}(\varkappa) & -\omega J
																						\end{array}
																			\right].
\]
The \emph{stability polynomial} of the relative equilibrium $\varkappa$ is, by definition:
\[ P_{\gamma}(\lambda)=\det(\lambda I_{4N}-L_{\gamma}(\varkappa))\]
Through simple manipulations, the above determinant can be written as
\[ P_{\gamma}(\lambda)=\det\left[\begin{array}{cc}
                           O_{2N} & -M^{-1}\\
                           D\nabla U_{\gamma}(\varkappa)+M(\lambda I_{2N}+\omega J)^2 & O_{2N}
                           \end{array}
                     \right],
\]
therefore
\begin{equation} 
P_{\gamma}(\lambda)=\det(M^{-1}D\nabla U_{\gamma}(\varkappa)+(\lambda^2-\omega^2) I_{2N}+2 \lambda \omega J). \label{secpol1}
\end{equation}
For the $N$-vortex problem, we have
\begin{align*}
 L_{-1}(\varkappa)&= D\mathbb{K}\nabla \mathcal{H}_{-1}(\varkappa) \\
          &= JM^{-1}D(\omega M\varkappa+\nabla H_{-1}(\varkappa))\\
          &=J(\omega I_{2N}+M^{-1}D\nabla H_{-1}(\varkappa)),
\end{align*}
hence 
\begin{equation} \label{secpol2}
P_{-1}(\lambda)=\det(M^{-1}D\nabla H_{-1}(\varkappa)+\omega I_{2N}+\lambda J).
\end{equation}

 We call the roots of $P_{\gamma}, P_{-1}$ the \emph{eigenvalues of} $\varkappa$. It is a basic fact that the linear stability of a relative equilibrium $\varkappa$ is determined by its eigenvalues and the existence of a basis formed by the eigenvectors of $L_{\gamma}(\varkappa)$. The analysis of the eigenvalues solely determines the spectral stability of $\varkappa$.
 
 Our goal in the present paper is to provide factorizations for $P_{\gamma}, P_{-1}$ when the configuration $\varkappa$ possesses symmetries. Our main tool is the representation theory of finite groups, specifically of the dihedral group $D_n$. We will not address the question of whether $\varkappa$ is or is not a relative equilibrium, however the literature on the $N$-body and $N$-vortex problems contains many related existence results. 
 
 
\section{Finite symmetric sets and their Displacements} \label{sec2}

 We briefly review some basic concepts and terminology. Let $\mathcal{M}$ be a metric space, and let $X$ be a subset of $\mathcal{M}$. A \emph{symmetry} of $X$ is an isometry of $\mathcal{M}$ which maps $X$ to itself. The symmetries of $X$ form a group $G$ under composition. We call $X$ a \emph{symmetric} (or $G$-\emph{symmetric}) \emph{set} if the group $G$ is nontrivial.

We will consider finite $G$-symmetric sets $X$ with $G$ finite. For instance, in the Euclidean plane, $G$ is finite if $X$ is finite and has at least two points.

 In order to describe the displacements of the points of a set $X$, we follow the approach in~\cite{s94} and consider a family of real vector spaces $\{E_x/x \in X\}$. The disjoint union of the elements of this family, $E=\cup_{x \in X} E_x$, is called a \emph{vector bundle} over $X$. 
The space $E_x$ is called the \emph{fiber} of $E$ over $x$, and a \emph{section} of $E$ is a function $\delta: X \longrightarrow E$ such that $\delta(x) \in E_x$. We endow the set of sections of $E$ with the structure of a real vector space in a natural way, and denote this space by $\sections$.
 
\begin{definition} \label{displacements}
The \emph{space of displacements} of $X$ in $E$ is the vector space of sections $\sections$. The elements of $\sections$ are called \emph{displacements} of $X$ in $E$.
\end{definition}

 Suppose a nontrivial group $G$ acts on $X$ (not necessarily through symmetries of $X$). A vector bundle $E$ over $X$ is called \emph{$G$-homogeneous} if each $g \in G$ acts on the fibers of $E$ as an isomorphism from $E_x$ to $E_{g\cdot x}$, where $\cdot$ denotes the action of $G$ on $X$. 
 
 A \emph{(linear) representation} of a group $G$ on a vector space $\vs$ is a homomorphism $\varrho$ between $G$ and the automorphism group $GL(\vs)$ of $\vs$. In the context of $G$-homogeneous vector bundles, we have the following definition.
 
\begin{definition} \label{canonical}
The \emph{canonical representation} of the group $G$ on the space $\sections$ is defined by
\[ [\varrho_E(g)(\delta)](x)=g \cdot \delta(g^{-1}\cdot x), \quad \forall g \in G, \ \forall \delta \in \sections,\ \forall x \in X,\]
where the dots indicate the actions of $G$ on $E$ and on $X$.
\end{definition}

\noindent The verification that $\varrho_E$ is indeed a representation is straightforward.
 
\begin{exam} \label{example1}
  Let $X \subset \reals^2$ be the set formed by the vertices of an equilateral triangle. The symmetry group of $X$ is the dihedral group $D_3=\{e,r,r^2,s,rs, r^2s\}$ considered as a group of rigid motions of the plane, where $r$ is a rotation about the center of the triangle by $\frac{2 \pi}{3}$, and $s$ is a reflection through a fixed median of the triangle. Let $E_{x}=T_{x}\reals^2$, for each $x \in X$, and write $E=T\reals^2\big|_{X}$, so $E$ is the restriction of the tangent bundle of $\reals^2$ to $X$. By viewing the fibers of $E$ as copies of $\reals^2$, we define the action $D_3$ on $E$ so that each $g \in D_3$ is a linear map which sends $v \in E_x$ to $g v \in E_{g x}$. $E$ is a $D_3$-homogeneous vector bundle over $X$, and $\varrho_E(g)(\delta)$ corresponds to rigidly moving the triangle $X$ \emph{together} with the displacement $\delta \in \sections$ according to $g$. 
\end{exam}

 In the next section we provide a broad generalization of the concepts featured in example~\ref{example1}.

 
 
\section{Finite $D_n$-symmetric subsets of $\reals^2$ and the canonical representation on the space of their planar displacements} \label{sec3}

 We call the \emph{standard representation} of $D_n$ the representation $\sigma:D_n \longrightarrow GL(\reals^2)$ defined by
\[ \sigma(r)=\left[\begin{array}{cc}
\cos 2\pi/n & -\sin 2 \pi/n \\
\sin 2 \pi/n & \cos 2 \pi/n
\end{array}
\right],  \quad \text{and} \quad \sigma(s) =\left[\begin{array}{cc}
-1 &0\\
0 & 1
\end{array}
\right].
\]
For a $D_n$-symmetric set $X \subset \reals^2$ (under the action given by $\sigma$), and for $E=T\reals^2\big|_X$, the corresponding canonical representation of $D_n$ on $\sections$  will be denoted by $\sigma^E$. The space $\sections$ is called the \emph{space of planar displacements of} $X$. 

 We describe the possible finite $D_n$-symmetric subsets of $\reals^2$. For each point $x \in X$, let $\mathcal{O}_x=\{g x\,/\,g\in D_n\}$ be the orbit of $x$ under $\sigma$~\footnote{The juxtaposition of $g \in D_n$ and $v \in \reals^2$ means $\sigma(g)(v)$.}.
 
\begin{prop} \label{rings}
For $n \geq 2$, a finite $D_n$-symmetric subset $X$ of $\reals^2$ with at least two points is the disjoint union of $a$ unit sets, $b$ regular $n$-gons and $c$ semiregular $2n$-gons, where $a=0,1$ and $b+c>0$. We call such $X$ a \emph{ring system of type} (a,b,c).
\end{prop}
\begin{proof} Recall $X$ is the disjoint union of the orbits of its points. We show there are three possible types of orbits. 

 Let $|X|$ be the cardinality of $X$. The point $O=\frac{1}{|X|}\sum_{x \in X} x$ is the barycenter of $X$. If $O \in X$, it is clear that $\mathcal{O}_O=\{O\}$. In particular, if $\ell$ is the line of $\reals^2$ fixed by the reflection $s\in D_n$, then $\ell$ contains $O$; indeed all reflection lines $r^{j/2} \ell$ must also go through $O$.  Besides, all rotations $r^j$ in $D_n$ are centered at $O$. Thus $O$ is the only point fixed by the action of $D_n$ on $X$. Let $a$ be 1 if $O \in X$ and 0 if $O \notin X$.
 
 Let $x \in X \setminus \{O\}$. We have two cases:
\[ \text{(1)} \ \ x \in \bigcup_{j=0}^{n-1} r^{j/2}\ell, \ \ \qquad \ \ \text{(2)}\ \ x \notin \bigcup_{j=0}^{n-1} r^{j/2}\ell.\]
In case (1), $\mathcal{O}_{x}$ is a regular $n$-gon centered at $O$. If $\mathcal{O}_x$ and $\mathcal{O}_{x'}$ are distinct regular $n$-gons in $X$, then $\mathcal{O}_x$ and $\mathcal{O}_{x'}$ are either homothetic or rotated relatively to one another by $\frac{\pi}{n}$. In case (2), $\mathcal{O}_x$ is a semiregular $2n$-gon. Two such semiregular $2n$-gons have alternating sides parallel and homothetic by the same factor.
\end{proof}

\begin{rem} \label{ring}
Proposition~\ref{rings} says that, for $n \geq 2$, $D_n$-symmetric sets are contained in the union of concentric circles of radii $\geq 0$, which motivates the terminology  \emph{ring system} applied to such sets.
\end{rem}

\begin{rem} \label{d1}
 A $D_1$-symmetric subset of $\reals^2$ consists of an arbitrary number of points on a line $\ell$ and an even number of points in $\reals^2 \setminus \ell$ pairwise equidistant to $\ell$.
\end{rem}

 Next we determine the structure of $\sigma^E$.
 
\begin{thm} \label{standardE}
 Let $X$ be a finite $D_n$-symmetric subset of $\reals^2$ and $\sigma$ be the standard representation of $D_n$ in $\reals^2$. Let $E=T\reals^2\big|_X$ and let $\sigma^E$ be the canonical representation of $D_n$ in the space of planar displacements of $X$, $\sections$. If $a, b$ and $c$ are as in proposition~\ref{rings}, then 
\[ \sigma^E \simeq a \sigma \oplus(b+2c)\regd.\]
\end{thm}
Before proving the above theorem, we explain the notation used in its statement. $\regd$ is the \emph{regular representation} of $D_n$; in general, the (left) regular representation $\reg$ of a group $G$ is the representation of $G$ in a (complex, real) vector space $\vs$ which possesses a basis $\{e_g\,/\, g\in G\}$ such that
\[ \reg(h)(e_g)=e_{hg}, \ \ \forall g,h \in G.\]
The direct sum with multiplicity on the righ-hand side of $\simeq$ means that we have a representation which, for each $g \in G$, is the direct sum of $a$ copies of the linear operator $\sigma(g)$ and $b+2c$ copies of $\regd(g)$. Finally the symbol $\simeq$ means that there exists an isomorphism $\Phi$ such that $\Phi\circ \sigma^E(g)= [a \sigma(g) \oplus(b+2c)\regd(g)]\circ \Phi$, for all $g \in D_n$.

\begin{prooff}
 Let $\orbit_1,\cdots,\orbit_l$ be the orbits of the action of $D_n$ on $X$. From the $D_n$-invariant decomposition $X=\cup_{i=1}^l \orbit_i$, we deduce the $D_n$-invariant decomposition of $E=T\reals^2\big|_X$, namely
\[ E=\bigcup_{i=1}^l T\reals^2\big|_{\orbit_i},\]
hence we obtain the $\sigma^E$-invariant decomposition of $\sections$, 
\[ \sections=\bigoplus_{i=1}^l \Gamma\left(T\reals^2\big|_{\orbit_i}\right),\]
and it suffices to look at the restrictions of $\sigma^E$ to each individual $\Gamma(T \reals^2\big|_{\orbit_i})$. 

 If the barycenter $O$ of $X$ is in $X$, its orbit $\orbit_O$ is just $\{O\}$ and its tangent plane is a copy of $\reals^2$ at $O$. From the definition of $\sigma^E$, the restriction of $\sigma^E$ to $\Gamma\left(T\reals^2\big|_{\orbit_O}\right)\cong T_O\reals^2 \cong \reals^2$ is just $\sigma$.
 
 
 If $X$ contains a regular $n$-gon, pick a vertex $x$ and index the vertices according to $x_j=r^{j}x$, $j=1,\cdots,n$. Let $s$ be the reflection about the line determined by $O$ and $x$, and choose a nonzero vector $v_e$ in $T_{x}\reals^2$ with the same direction and orientation as the vector $\overrightarrow{xx_1} \in \reals^2$. Define the planar displacement $\epsilon$ of the $n$-gon $\orbit_x$ as $v_e$ at $x$ and the zero vector at $x_j$, $j\neq n$. Let $\epsilon_g=\sigma^E(g)(\epsilon)$. The set $\{\epsilon_g\,/\,g \in D_n\}$ is a basis of $\Gamma(T \reals^2\big|_{\orbit_x})$ and, with respect to this basis, $\sigma^E$ is given by $\regd$. Figure~\ref{f-1} illustrates the case $n=5$.
 
 If $X$ contains a semiregular $2n$-gon, we apply a similar procedure. Pick two neighboring vertices $x,x'$ and let $x_j=r^{j}x$, $x'_j=r^{j}x'$, $j=1,\cdots,n$. Take $s$ as the reflection about the perpendicular bisector of the segment $\overline{xx'}$, and select $v_e \in T_{x}\reals^2$ with the same direction as, and opposite orientation to, $\overrightarrow{x'x}$, and $v_e' \in T_{x}\reals^2$ with the same direction and orientation as $\overrightarrow{x'_{n-1}x}$. Define displacements $\epsilon, \epsilon'$ as in the previous paragraph, using the vectors $v_e$ and $v_e'$, respectively. Form the sets $\{\sigma^E(g)(\epsilon)\,/\, g \in D_n\}$ and $\{\sigma^E(g)(\epsilon')\,/\, g \in D_n\}$. The union of these sets is a basis of $\Gamma(T \reals^2\big|_{\orbit_x})$ with respect to which $\sigma^E$ and $\regd \oplus \regd$ have the same matrix representations. \hfill $\square$
\end{prooff}

\begin{figure}[h]
\begin{center}
\scalebox{.3}{\includegraphics{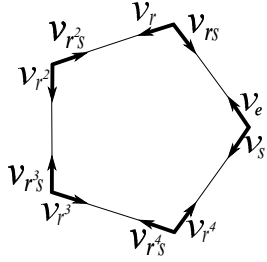}}
\caption{The vectors $v_g=\epsilon_g(x_j)$, $g\in D_5$, $j=1,\cdots,5$, leading to the construction of the isomorphism $\sigma^E \simeq \rho^{\text{reg}}_{D_5}$ for the regular pentagon.}
\label{f-1}
\end{center}
\end{figure}

 We conclude this section with the definition of a representation of $D_n$ on $\reals^{2N}$ which is isomorphic to $\sigma^E$ in a natural way. 
 
 Let $\varkappa=(x_1,\cdots,x_N) \in \reals^{2N}$. We say that $\varkappa$ is a $D_n$\emph{-symmetric configuration} if the corresponding set $X=\{x_1,\cdots,x_N\}\subset \reals^2$ is $D_n$-symmetric. The group $D_n$ acts by permutations of the $x_j \in X$, so it makes sense to write $g \cdot x_j=x_{g(j)}$. Using this notation, and identifying the fibers $E_x$ of $E=T\reals^2\big|_X$ with $\reals^2$, the action of $D_n$ on $E$ takes a vector $v\in E_{x_j}$ to the vector $gv \in E_{x_{g(j)}}$. We can identify the real $2N$-dimensional vector space $\sections$ with $\reals^{2N}$ by sending a displacement $\delta$ to the vector $(\delta(x_1),\cdots,\delta(x_N))$. From the definition of $\sigma^E$, this isomorphism sends $\sigma^E(g)(\delta)$ to $(\cdots,g\delta(x_j),\cdots)$, where $g\delta(x_j)$ occupies the $g(j)$-th pair of coordinates. Therefore we have that $\sigma^E$ is isomorphic to the representation $\sigma_X:D_n \longrightarrow GL(\reals^{2N})$ defined by:
\begin{equation} \label{repconf}
 \sigma_X(g)(w)=(g w_{g^{-1}(1)},\cdots,gw_{g^{-1}(N)}),\ \ \forall g \in D_n,
\end{equation}
where $w=(w_1,\cdots,w_N)$. Observe that, for a $D_n$-symmetric configuration $\varkappa$, we have that
\[ \sigma_X(g)(\varkappa)=\varkappa, \ \ \forall g\in D_n.\]
So the symmetry of the set $X \subset \reals^2$  implies that its originating configuration $\varkappa \in \reals^{2N}$ is a fixed point of the representation $\sigma_X$. 

\section{Symmetric Functions and Equivariant Hessians} \label{sec4}
 
  Let $F:\mathfrak{O}\longrightarrow \reals$ be a function defined on the open subset $\mathfrak{O}\subset \reals^{2N}$. Suppose a group $G$ acts on $\reals^{2N}$ so that $g \cdot \varkappa \in \mathfrak{O}$ for all $g \in G, \varkappa \in \mathfrak{O}$. We say that $F$ is \emph{symmetric} if:
\begin{equation*} \label{symmf}
 F(g \cdot \varkappa)=F(\varkappa), \ \ \forall g \in G, \forall \varkappa \in \mathfrak{O}.
\end{equation*}
Suppose $F$ is of class $C^2$ and $G$ acts through a linear representation $\varrho$. After differentiating both sides of the above equation twice, we obtain
\begin{equation} \label{hesssymm}
 \varrho(g)^T D^2F(\varrho(g)(\varkappa)) \varrho(g)=D^2F(\varkappa).
\end{equation}


\begin{prop} \label{hess1} Let $\varkappa \in \reals^{2N}$ be a $D_n$-symmetric configuration. If $F$ is $D_n$-symmetric and of class $C^2$, and the action of $G=D_n$ is given by the representation $\sigma_X$ defined by equation~\eqref{repconf}, then
\begin{equation} \label{hessrho}
D^2F(\varkappa)\sigma_X(g)=\sigma_X(g) D^2F(\varkappa), \ \ \forall g \in D_n.
\end{equation}
\end{prop}
\begin{proof} At the end of the last section, we saw that a $D_n$-symmetric configuration $\varkappa$ is fixed by $\sigma_X$. We apply formula~\eqref{hesssymm}, together with the fact that $\sigma_X(g)^T=\sigma_X(g)^{-1}$, which holds since $\sigma_X(g)$ preserves the inner product of $\reals^{2N}$, for all $g \in D_n$.
\end{proof}
 
 
 We apply proposition~\ref{hess1} to the $N$-body and $N$-vortex problems. Let $\Delta=\{\varkappa \in \reals^{2N}\,/\,x_i = x_j, \text{ for some } i\neq j\}$. Consider the function $F:\reals^{2N}\setminus \Delta \longrightarrow \reals$ given by either the potential energy
\[ U_{\gamma}(\varkappa)=\frac{1}{2\gamma+2}\sum_{i <j}m_im_j ||x_i-x_j||^{2\gamma+2}, \ \ \text{for some} \ \ \gamma \neq -1,\]
or the Hamiltonian
\[ H_{-1}(\varkappa)= -\sum_{i<j}m_im_j \ln ||x_i-x_j||.\]
$F$ is symmetric with respect to $\sigma_X$ as long as $m_{g(i)}=m_i$, $i=1,\cdots,N$, for all $g \in D_n$. We have the following definition.

\begin{definition} \label{equivariant} Let $\vs$ be a vector space and $\varrho:G \longrightarrow GL(\vs)$ be a linear representation. A linear operator $L:\vs  \longrightarrow \vs$ is \emph{equivariant} with repect to $\varrho$, or \emph{$\varrho$-equivariant}, if
\[ \varrho(g) \circ L = L \circ \varrho(g), \qquad \forall g \in G.\] 
\end{definition}
\noindent The term equivariant extends to matrices in a natural manner.

\begin{prop} \label{hess2} Let $\varkappa$ be a $D_n$-symmetric configuration and recall the mass matrix $M=\emph{diag}(m_1,m_1,\cdots,m_N,m_N)$. 
Suppose $m_{g(i)}=m_i$ , $i=1,\cdots,N$, for all $g \in D_n$. The matrices $M^{-1}D\nabla U_{\gamma}(\varkappa)$ and $M^{-1}D\nabla H_{-1}(\varkappa)$ are $\sigma_X$-equivariant.
\end{prop}
\begin{proof} Let $F=U_{\gamma}$, $\gamma \neq -1$, or $F=H_{-1}$. The product of equivariant matrices is clearly equivariant, and since the matrices $D\nabla F(\varkappa)$ and $D^2F(\varkappa)$ coincide, proposition~\ref{hesssymm} implies the former matrix commutes with $\sigma_X (g)$, for all $g \in D_n$. We show that $M$ also commutes with all the $\sigma_X(g)$. Indeed, if $w=(w_1,\cdots,w_N)\in \reals^{2N}$ and $g \in D_n$, we have that
\begin{align*}
 M\sigma_X(g)(w)&=(m_1gw_{g^{-1}(1)}\cdots,m_N g w_{g^{-1}(N)})\\
            &=(g m_1 w_{g^{-1}(1)},\cdots,g m_N w_{g^{-1}(N)})\\
            &=\sigma_X(g)(Mw),
\end{align*}
as long as $m_{g(i)}=m_i$, for all $i=1,\cdots,N$ and all $g\in D_n$.
           
\end{proof}


\section{Block-diagonal Forms of Equivariant Operators} \label{perisot}

 In the following paragraphs we explain how an equivariant linear operator can be put in block-diagonal form through the usage of symmetry. This is a classical application of group representation theory which goes back to the early work of E. P. Wigner on Group Theory and Quantum Mechanics~\cite{mw93}. The fundamental concepts are introduced in the next paragraph. References~\cite{fs94} and~\cite{jps77} contain detailed expositions of the concepts and results presented in this section. We will consider only representations over the complex field.
 
 A \emph{subrepresentation} of a representation $\varrho:G \longrightarrow GL(\vs)$ is a representation obtained by restricting each $\varrho(g)$ to a $\varrho$-invariant subspace of $\vs$. $\varrho$ is said to be \emph{irreducible} if its only subrepresentations are trivial, i.e., if $\{0\}$ and $\vs$ are the only $\varrho$-invariant subspaces, and $\varrho$ is \emph{completely reducible} if it is the direct sum of irreducible subrepresentations. A basic theorem from representation theory states that the number of isomorphism classes of irreducible representations of a finite group is finite (and equal to the number of conjugacy classes in the group). A fundamental theorem asserts that every representation of a finite group is completely reducible.  

\subsection{Invariant decompositions and block diagonalizations} \label{perirred}
 Let $\varrho_1,\cdots,\varrho_t$ be a complete list of irreducible representations of a finite group $G$ up to isomorphisms. For every $k=1,\cdots,t$, let $(r_{ij}^{(k)})$, $i,j=1,\cdots, d_k$ be a unitary matrix representation isomorphic to $\varrho_k$. The integer $d_k$ is called the \emph{degree} of $\varrho_k$. Given a representation $\varrho:G \longrightarrow GL(\vs)$, define the operators  
\[ p_{ij}^{(k)}= \frac{d_k}{|G|} \sum_{g \in G} r_{ji}^{(k)}(g^{-1})\varrho(g), \quad i,j=1,\cdots, d_k, \ \ k=1,\cdots,t.\]
Schur orthogonality relations for the entry functions $r_{ij}^{(k)}:G \longrightarrow \mathbb{C}$ imply the identities:
\begin{equation} \label{algebra}
p_{ij}^{(k)}\circ p_{i'j'}^{(k')}=\begin{cases}
																											p_{ij'}^{(k)},& \text{if } j=i',k=k',\\
																											0, & \text{if } j \neq i', \text{or } k \neq k'.
																											\end{cases}
\end{equation}
The next theorem states some remarkable properties of the $p_{ij}^{(k)}$ and introduces relevant terminology.

\begin{thm} \label{maintheo} Suppose $k$ is fixed. The operators $p_{ij}^{(k)}$ have the following properties:
\begin{enumerate}
\item The operators $p_{ii}^{(k)}$ are \emph{projections} whose images $\vs_{i}^{(k)}$ are such that:
\begin{equation} \label{equivdec}
 \vs^{(k)}=\bigoplus_{i=1}^{d_k} \vs_{i}^{(k)}
\end{equation}
is the space of a subrepresentation of $\varrho$ whose decompositions into irreducible subrepresentations consist of precisely $\mu_k=\emph{dim} \vs^{(k)}_1$ copies (up to isomorphism) of the irreducible representation $\varrho_k$. We have that $\vs=\oplus_{k=1}^t \vs^{(k)}$, and we call $\vs^{(k)}$ the $\varrho_k$-\emph{isotypic component} of $\vs$. The number $\mu_k$ is called the \emph{multiplicity} of $\varrho_k$ in $\varrho$.
\item For every $i, j$, $p_{ij}^{(k)}$ is an isomorphism between $\vs_{j}^{(k)}$ and $\vs_{i}^{(k)}$, and $p_{ij}^{(k)}$ vanishes on subspaces complementary to $\vs_j^{(k)}$ in $\vs$. In particular, $\mu_k=\emph{dim} \vs^{(k)}_i$, for all $i$. We call each $p_{ij}^{(k)}$, with $i\neq j$, a \emph{transfer isomorphism}.
\end{enumerate}
\end{thm}

  The operators $p_{ij}^{(k)}$ may have useful additional properties when $\vs$ is a complex inner product space.
\begin{lemma} \label{ortproj}
Suppose $\vs$ has an inner product $\langle,\rangle$ which is preserved by $\varrho$, i.e.,
\[ \langle \varrho(g)(v),\varrho(g)(w) \rangle=\langle v,w\rangle, \ \ \forall g \in G, \ v,w \in \vs.\]
The projections $p_{ii}^{(k)}:\vs \longrightarrow \vs_{i}^{(k)}$ are orthogonal projections and the transfer isomorphisms $p_{ij}^{(k)}:\vs_{j}^{(k)}\longrightarrow \vs_{i}^{(k)}$ are isometries. 
\end{lemma}
\begin{proof} Fix $k$ and choose a basis for $\vs$. Let $B$ be the matrix of $\langle,\rangle$ with respect to the chosen basis. We have that
\begin{align*}
 \overline{[p_{ij }^{(k)}]^T} B&=\left(\frac{d_k}{|G|}\sum_{g \in G}\overline{r_{j i}^{(k)}(g^{-1})}\,\overline{\varrho(g)^T}\right)B\\
                                         &=\frac{d_k}{|G|}\sum_{g \in G}\overline{r_{j i}^{(k)}(g^{-1})}B \varrho(g^{-1})\\
                                         &=B\left(\frac{d_k}{|G|}\sum_{g \in G}r_{ij}^{(k)}(g) \varrho(g^{-1})\right)=Bp_{ji}^{(k)}.
\end{align*}
Notice we used the identity $\overline{r_{j i}^{(k)}(g^{-1})}=r_{ij}^{(k)}(g)$, which holds since $(r_{ij}^{(k)})$ is a unitary matrix. The identities~\eqref{algebra} and part (2) of theorem~\ref{maintheo} imply that the inverse of $p_{ij}^{(k)}:\vs_{j}^{(k)}\longrightarrow \vs_{i}^{(k)}$ is the restriction of $p_{j i}^{(k)}$ to $\vs_{i}^{(k)}$.
\end{proof}

  The decomposition~\eqref{equivdec} in part (1) of theorem~\ref{maintheo} plays the key role in block diagonalizing equivariant operators (see definition~\ref{equivariant}). Suppose $L:\vs \longrightarrow \vs$ is a $\varrho$-equivariant operator. Fix $k$ and consider a basis $\{v_{1j}^{(k)}\}_{j=1}^{\mu_k}$ of $\vs_{1}^{(k)}$. Using the transfer isomorphisms $p_{i1}^{(k)}$, we obtain the bases $\{v_{ij}^{(k)}=p_{i1}^{(k)}(v_{1j}^{(k)})\}_{j=1}^{\mu_k}$ for each $\vs_{i}^{(k)}$, $i=1,\cdots,d_k$. The equivariance of $L$ implies that $L$ commutes with each projection $p_{ii}^{(k)}$, so $L$ leaves each $\vs_{i}^{(k)}=\text{Image}(p_{ii}^{(k)})$ invariant. Moreover, since $L$ also commutes with the transfer isomorphisms, the matrix representation of $L$ with respect to the basis $\{v_{ij}^{(k)}\}_{j=1}^{\mu_k}$ is the same for every $i=1,\cdots,d_k$. By repeating the construction just described for each $k=1,\cdots,t$, and observing that the isotypic decomposition $\vs=\oplus_{k=1}^t \vs^{(k)}$ is $L$-invariant, we obtain a matrix for $L$ in block-diagonal form, 
\[\text{diag}(L_1,\cdots,L_1,\cdots,L_t,\cdots,L_t),\]
with $d_1$ blocks $L_1$ of size $\mu_1 \times \mu_1$, $\cdots$, $d_t$ blocks $L_t$ of size $\mu_t \times \mu_t$. 

\begin{rem} \label{regrepequiv}
 Using character theory, it can be shown that the regular representation of a finite group $G$ is isomorphic to the direct sum of copies of $\varrho_j$ with multiplicities equal to the corresponding degrees, that is,
 $\reg \simeq \bigoplus_{k=1}^t d_k \varrho_k.$
Thus every $\reg$-equivariant operator $L$ has a block-diagonal form consisting of $d_1$ blocks $L_1$ of size $d_1\times d_1$, $\cdots$, $d_t$ blocks $L_t$ of size $d_t\times d_t$.
\end{rem}

\subsection{Decomposition of $\bm{\sigma^E}$} \label{sec6}

 We apply the results from section~\ref{perirred} to $\sigma^E:D_n \longrightarrow GL(\sections)$, where $\sections$ is the space of planar displacements of a ring system $X$ of type $(a,b,c)$. 
 
 The following table contains a complete list of matrix representatives of the isomorphism classes of irreducible representations of the groups $D_n$, for $n$ even. Let $\theta=2 \pi/n$.
 
 
\vspace{.2cm}
\begin{center}
\begin{tabular}{c|cc} \label{characterDn}
 &   $r^j$ & $r^js$\\
\hline
$\tau$ & $1$ & $1$ \\
$\alpha$ & $1$  &$-1$\\
$\phi$ & $(-1)^j$ & $(-1)^j$\\
$\psi$ & $(-1)^j$ & $(-1)^{j+1}$\\
$\begin{array}{ccc}
&\varrho_k&\\
&k=1,\cdots,\frac{n}{2}-1
\end{array}$  & $ \left[\begin{array}{cc}
\cos kj\theta & -\sin kj\theta \\
\sin kj\theta & \cos kj\theta
\end{array}
\right]$ & $ \left[\begin{array}{cc}
\cos kj\theta & \sin kj\theta \\
\sin kj\theta & -\cos kj\theta
\end{array}
\right]$
\end{tabular}
\end{center}
\vspace{.2cm}
The representations $\tau$ and $\alpha$ are known as the \emph{trivial representation} and the \emph{alternating representation}, respectively. If $n$ is odd, the representations $\phi$ and $\psi$ are absent, and $k$ ranges from 1 to $\frac{n-1}{2}$. Notice that $\varrho_1$ is just the standard representation $\sigma$. 

 Next we list the projections and transfer isomorphisms for $\sigma^E$. For the trivial representation, we have the projection:
\[ p^{(\tau)}=\frac{1}{2n}\sum_{j=1}^n (\sigma^E(r^j)+\sigma^E(r^js))=\frac{1}{2n}\sum_{j=1}^n \sigma^E(r^j)(\sigma^E(e)+\sigma^E(s)).\]
In order to simplify our notation, let us henceforth omit $\sigma^E$, so that we have
\begin{equation} \label{ptau}
 p^{(\tau)}=\left[\frac{1}{2n}\sum_{j=1}^n r^j\right](e+s).
\end{equation}
We denote the operator between brackets by $c_0(r)$. More generally, for each $k$ define
\begin{equation} \label{cyclic}
c_k(r)=\frac{1}{2n}\sum_{j=1}^n \cos\frac{2 \pi kj}{n}r^j, \quad s_k(r)=\frac{1}{2n}\sum_{j=1}^n \sin\frac{2 \pi kj}{n}r^j
\end{equation}
Thus we have the projections \label{projs}
\[ p^{(\tau)}=c_0(r)(e+s), \ \ \ p^{(\alpha)}=c_0(r)(e-s), \ \ \ p^{(\phi)}=c_0(-r)(e+s), \ \ \ p^{(\psi)}=c_0(-r)(e-s),\]
and the projections and transfer isomorphisms \label{transfers} associated with the representations $\varrho_1,\cdots,\varrho_t$ are 
\[ p_{11}^{(k)}=2 c_k(r)(e+s),  \, \, p_{22}^{(k)}=2c_k(r)(e-s),  \, \, p_{12}^{(k)}=2 s_k(r)(-e+s), \, \, p_{21}^{(k)}=2 s_k(r)(e+s).\]

 The above projections, when applied to the space of displacements $\sections$, produce the isotypic decomposition of $\sections$, which for $n$ even looks like:
\begin{equation} \label{isotypical}
 \sections = \vs^{(\tau)} \oplus \vs^{(\alpha)} \oplus \vs^{(\phi)}\oplus \vs^{(\psi)}\oplus\left( \bigoplus_{k=1}^t \vs^{(k)}\right)
\end{equation}
where
\[ \vs^{(\varrho)}=\text{Image}\left(p^{(\varrho)}\right), \ \ \varrho=\tau,\alpha,\phi,\psi, \quad \vs^{(k)}=\vs_1^{(k)} \oplus \vs_2^{(k)},\] 
with $\vs_i^{(k)}=\text{Image}\left(p_{ii}^{(k)}\right)$, $i=1,2,\ k=1,\cdots,n/2-1$ (cf. theorem~\ref{maintheo}). From theorem~\ref{standardE}, remark~\ref{regrepequiv} and the previous table, the multiplicities of the irreducible representations of $D_n$ in $\sigma^E$ are
\begin{equation} \label{multip}
 \mu_{\tau}=\mu_{\alpha}=\mu_{\phi}=\mu_{\psi}=b+2c, \ \ \ \mu_{\sigma}=a+2(b+2c), \ \ \ \mu_k=2(b+2c), \ \ k>1.
\end{equation}
Thus, using the method described at the end of section~\ref{perirred}, we can construct bases for $\sections$ with respect to which any $\sigma^E$-equivariant operator can be put in block-diagonal form with four (resp. two) blocks of size $(b+2c)\times (b+2c)$, if $n$ is even (resp. odd), two identical blocks of size $a+2(b+2c)\times a+2(b+2c)$, and $n/2-2$ (resp. $(n-3)/2$)  pairs of identical blocks of size $2(b+2c)\times2(b+2c)$.

 

\section{Additional structures and $J$ as an operator on $\sections$} \label{sec7}

 Let $J$ and $M$ be as in equation~\eqref{jem}. Before proceeding to the applications of section~\ref{perisot}, we introduce some useful structures in $\sections$. We also need to examine how $J$ relates to the canonical representation $\sigma^E$ and its invariant decomposition~\eqref{isotypical}. 
 
 Throughout the remainder of this paper, we suppose, as in the end of section~\ref{sec4}, that every $D_n$-symmetric configuration $\varkappa \in \reals^{2N}$ of point masses has a \emph{symmetric mass distribution}, i.e., $m_{g(i)}=m_{i}$ for every $g\in D_n$ and every $i=1,\cdots,N$.

 We use the natural isomorphism $\sections \simeq \reals^{2N}$ to define an (indefinite) inner product $\langle,\rangle_M$ on $\sections$ corresponding to the inner product defined by $M$ on $\reals^{2N}$, which we will also denote by $\langle,\rangle_M$.  As  before, let $\orbit_1,\cdots,\orbit_l$ be the orbits of $D_n$ in $X \subset \reals^2$, the point set associated with the configuration $\varkappa$. We have that
\begin{equation} \label{inner}
 \langle \delta,\epsilon \rangle_M=\sum_{j=1}^l m_j\left(\sum_{x \in \orbit_j}\delta(x) \cdot \epsilon(x)\right),\ \ \forall \delta,\epsilon \in \sections
\end{equation}
where $\cdot$ denotes the inner product inherited by each $T_x\reals^2$ from $\reals^2$. 

 
\begin{lemma} \label{rhoM}
 Consider $\langle,\rangle_M$ on $\reals^{2N}$ and the representation $\sigma_X$ defined by~\eqref{repconf}. For every $g \in D_n$, the isomorphism $\sigma_X(g)$ is an isometry.
\end{lemma}
\begin{proof} Let $v=(v_1,\cdots,v_N)$ and $w=(w_1,\cdots,w_N)$. For every $g \in D_n$, we have
\begin{align*}
 \langle \sigma_X(g)(v),\sigma_X(g)(w)\rangle_M &=\sum_{i=1}^N m_i (gv_{g^{-1}(i)} \cdot gw_{g^{-1}(i)})= \sum_{i=1}^N m_i (v_{g^{-1}(i)} \cdot w_{g^{-1}(i)})\\
                                        &=\sum_{i=1}^N m_{g^{-1}(i)} (v_{g^{-1}(i)} \cdot w_{g^{-1}(i)})=\langle v, w\rangle_M.                
\end{align*}
\end{proof}

\begin{prop} \label{standardMprojtrans}
Consider $\sections$ endowed with an inner product $\langle,\rangle_M$ as in~\eqref{inner}. The canonical representation $\sigma^E$ on $\sections$ leaves $\langle,\rangle_M$ invariant. The projections $p^{(\tau)}$, $p^{(\alpha)}$, $p^{(\phi)}$, $p^{(\psi)}$, $p_{11}^{(k)}$, $p_{22}^{(k)}$ are orthogonal and the transfer isomorphisms $p_{12}^{(k)},p_{21}^{(k)}$ are isometries, $k=1,\cdots,t$. As a consequence, the subspaces in the isotypical decomposition~\eqref{isotypical}, as well as each $\vs_1^{(k)}$ and $\vs_2^{(k)}$, are mutually $M$-orthogonal.
\end{prop}
\begin{proof} Recall $\sigma^E \simeq \sigma_X$. All claims follow from lemmas~\ref{ortproj} and~\ref{rhoM}, as well as the definitions of the projections and transfer isomorphisms in subsection~\ref{sec6}.
\end{proof}

 Next we examine $J$ in more detail. Viewed as an operator on $\sections$, $J$ is given $(J\delta)(x)=R(\pi/2)\delta(x)$, where $R(\pi/2)$ is the rotation by $\pi/2$ of $T_x\reals^2$. By definition, $(\sigma^E(r)(\delta))(x)=r\cdot\delta(r^{-1}\cdot x)$, and $(\sigma^E(s)(\delta))(x)=s\cdot\delta(s\cdot x)$, for all $x \in X$, so we have that
\[ J \circ \sigma^E(r)=\sigma^E(r)\circ J, \quad \text{and} \quad J \circ \sigma^E(s)=-\sigma^E(s)\circ J.\]
or, in simplified notation,
\begin{equation} \label{Jrs}
 Jr=rJ \quad \text{and} \quad Js=-sJ.
\end{equation}

\begin{prop} \label{comJ} Consider the operator $J:\sections \longrightarrow \sections$ and the projections and transfer isomorphisms associated with $\sigma^E$. The following relations hold
\begin{equation} \label{projsJ}
Jp^{(\tau)}=p^{(\alpha)}J, \ \ \, Jp^{(\alpha)}=p^{(\tau)}J, \ \ \, J p^{(\phi)}=p^{(\psi)}J,  \, \ \ J p^{(\psi)}=p^{(\phi)}J,
\end{equation}
and, for $k=1,\cdots,t$,
\begin{equation} \label{transJ}
J p_{11}^{(k)}=p_{22}^{(k)}J, \ \ \, J p_{22}^{(k)}=p_{11}^{(k)}J, \ \ \, Jp_{12}^{(k)}=-p_{21}^{(k)}J, \ \ \, Jp_{21}^{(k)}=-p_{12}^{(k)}J.
\end{equation}
\end{prop}
\begin{proof} Equations~\eqref{projsJ} and~\eqref{transJ} follow directly from equations~\eqref{Jrs} and the expressions for the projections and transfer isomorphisms.
\end{proof}

\begin{cor} \label{Jinv} The direct sums of isotypic components $\vs^{(\tau)}\oplus\vs^{(\alpha)}$, $\vs^{(\phi)}\oplus\vs^{(\psi)}$ and the isotypic components $\vs^{(k)}=\vs_1^{(k)}\oplus \vs_2^{(k)}$, in the decomposition~\eqref{isotypical} are $J$-invariant subspaces for all $k$. Furthermore, the subspaces in each of these direct sums are interchanged by $J$:
\[ J\vs^{(\tau)}=\vs^{(\alpha)}, \ \ J\vs^{(\alpha)}=\vs^{(\tau)}, \ \ J\vs^{(\phi)}=\vs^{(\psi)}, \ \ J\vs^{(\psi)}=\vs^{(\phi)},\]
and
\[ J\vs_1^{(k)}=\vs_2^{(k)}, \ \ J\vs_2^{(k)}=\vs_1^{(k)}, \ \ \ \forall k.\]
\end{cor}

\subsection{Symplectic Structures on ${\bm \sections}$}  Consider $\sections$ endowed with $\langle,\rangle_M$. Since $J^TM=-MJ$, we may use $J$ to endow $\sections$ with the symplectic form $\Omega_M$ given by
\begin{equation} \label{symplecticJ}
 \Omega_M(\delta,\epsilon)=\langle \delta,J\epsilon\rangle_M, \quad \delta,\epsilon \in \sections.
\end{equation}
Equations~\eqref{projsJ} and~\eqref{transJ} imply that the direct sums in the statement of corollary~\ref{Jinv} are Lagrangean decompositions with respect to any $\Omega_M$. For instance, $\vs^{(\tau)}=\text{Im}(p^{(\tau)})$ is $\Omega$-isotropic in $\vs^{(\tau)}\oplus\vs^{(\alpha)}$ since, from proposition~\ref{standardMprojtrans},
\[ \Omega_M(p^{(\tau)}(\delta),p^{(\tau)}(\epsilon))=\langle p^{(\tau)}(\delta),p^{(\alpha)}(J \epsilon)\rangle_M=0.\]
In an analogous manner, we verify that the restrictions of the transfer isomorphisms $p_{12}^{(k)}$ and $p_{21}^{(k)}$, $k=1,\cdots,t$, to each isotypic component $\vs^{(k)}$ are Hamiltonian operators. In the sequel, we will determine bases for $\vs^{(\tau)}\oplus\vs^{(\alpha)}$, $\vs^{(\phi)}\oplus\vs^{(\psi)}$ and $\vs^{(k)}$, $k=1,\cdots,t$ which are $M$-orthogonal and, if $M$-normalization is allowed, such bases may also be made $\Omega_{M}$-symplectic.


\section{Factorization of the Stability Polynomial of a Ring System with a Single Ring} \label{1orbit}

  Let $\varkappa \in \reals^{2N}$ be a $D_n$-symmetric configuration with a symmetric mass distribution, and $A$ be either the matrix $M^{-1}D\nabla U_{\gamma}(\varkappa)$ or the matrix $M^{-1}D\nabla H_{-1}(\varkappa)$. From equations~\eqref{secpol1} and ~\eqref{secpol2}, the stability polynomials of the $N$-body and $N$-vortex problems are expressed as the determinants of linear combinations of the matrices $A$, $J$ and the identity matrix.
We consider the operators on $\sections$ corresponding to $A$ and $J$ through the natural isomorphism $\reals^{2N}\simeq \sections$, keeping the notations $A,J$ for the respective operators on $\sections$. We adopt the basic strategy of Moeckel~\cite{m95}, and recast problem of factorizing $P_{\gamma}(\lambda)$, $P_{-1}(\lambda)$ as the problem of finding decompositions of $\sections$ into subspaces simultaneously invariant by $A$ and $J$ and having dimensions as small as possible. Notice a $J$-invariant subspace is necessarily even-dimensional.





 According to section~\ref{perisot}, $A$ can be block-diagonalized via the determination of a basis for $\sections$ which realizes the decomposition~\eqref{isotypical}. Recall the subspaces in~\eqref{isotypical} are images of the projections associated with $\sigma^E$. Proposition~\ref{comJ} describes how $J$ relates to the projections and transfer isomorphisms associated with $\sigma^E$. 

 


 We proceed to the block diagonalization of operators of the form $ A +\kappa J$, with $\kappa$ a scalar variable. Since all the projections and transfer isomorphisms associated with the representation $\sigma^E$, as well as the operator $J$, respect the orbit decomposition $\sections= \Gamma\left(\bigoplus_{i=1}^l T\reals^2\big|_{\mathcal{O}_i}\right)$, it is natural to start with the determination of bases which block diagonalize $A+\kappa J$ for each possible orbit type.  In the proof of proposition~\ref{rings}, we verified that individual orbits may consist of a single point, the vertices of a regular $n$-gon or the vertices of a semiregular $2n$-gon. For a singleton $X=\{O\}$, we have $\sections=\Gamma(T_O\reals^2)\cong \reals^ 2$, and $\sigma^E\cong\sigma$, the standard representation of $D_n$. Thus, when $X=\{O\}$, we can simply pick the basis of $\sections$ corresponding to the canonical basis of $\reals^2$. In the sequel, we consider the regular $n$-gons and the semiregular $2n$-gons.
 
\begin{rem} \label{simantisim} As noted in~\cite{m95}, $A^TM=MA$ and $J^TM=-MJ$ imply that, with respect to the inner product $\langle,\rangle_M$ defined in~\eqref{inner}, $A$ is symmetric and $J$ is antisymmetric. In particular, if $S$ is a subspace of $\sections$ which is both $A$ and $J$-invariant, then its $M$-orthogonal subspace $S^{\perp_M}=\{\epsilon\,/\,\langle \delta,\epsilon \rangle_M=0,\ \forall v\in S\}$ is also $A$ and $J$-invariant. When all masses have the same sign, $S^{\perp_M}$ is the $M$-orthogonal complement of $S$, and if $S$ is two-dimensional and contains an eigenvector $v$ of $A$, then $Jv$ is also an eigenvector of $A$. 
\end{rem}

\begin{rem} \label{nbody-nvortex} Suppose $\varkappa$ is a relative equilibrium of an $N$-body problem. The homogeneity of the corresponding potential function $U_{\gamma}$, together with the invariance of $U_{\gamma}$ with respect to rotations and translations, imply that $\varkappa$, $J\varkappa$, and displacements which translate $\varkappa$ in the horizontal and vertical directions are all eigenvectors of $A$, as direct calculations show~\cite{m95}. A similar remark holds for the $N$-vortex problem, see~\cite{gr13}, with the exclusion of the eigenvector $\varkappa$.
\end{rem}

\begin{rem} \label{D2esp}
 In this section and the next we will suppose that $n>2$. The case $n=2$ corresponds to arbitrary two-point sets $X \subset \reals^2$ and any such $X$ has symmetry group $D_2$. In section~\ref{sec10}, lemma~\ref{redstan}, we observe that the standard representation of $D_2$ is actually reducible, and this fact sets the case $n=2$ apart from the cases $n>2$.
\end{rem}

\subsection{Block diagonalization for regular $\bm{n}$-gons} \label{sec8.1}

  Let $X$ be the set of vertices of a regular $n$-gon. The corresponding subspaces $\vs^{(\tau)}$, $\vs^{(\alpha)}$, $\vs^{(\phi)}$, $\vs^{(\psi)}$ are all one-dimensional. Figure~\ref{f0} illustrates displacements which generate these subspaces. Notice each such displacement is an eigenvector of $A$, indeed of \emph{every} $\sigma^E$-equivariant operator on $\sections$.
We will denote the respective generators of $\vs^{(\tau)}$ and $\vs^{(\phi)}$ by $\varkappa$ and $\delta^{(\phi)}$. From corollary~\eqref{Jinv}, it is clear that $J\varkappa$ and $J \delta^{(\phi)}$ are generators of $\vs^{(\alpha)}$ and $\vs^{(\psi)}$, respectively. The matrix of an operator $ A +\kappa J$ with respect to the basis $\{J\varkappa,\varkappa,J\delta^{(\phi)},\delta^{(\phi)}\}$ has the form
\[ \left[\begin{array}{cccc}
         * & \kappa &&\\
         -\kappa& * &&\\
         & & * & \kappa\\
         && -\kappa & *
         \end{array}
   \right]
\]
where the $*$ substitute for the eigenvalues of $A$.  

\begin{figure}[h]
\begin{center}
\scalebox{.5}{\includegraphics{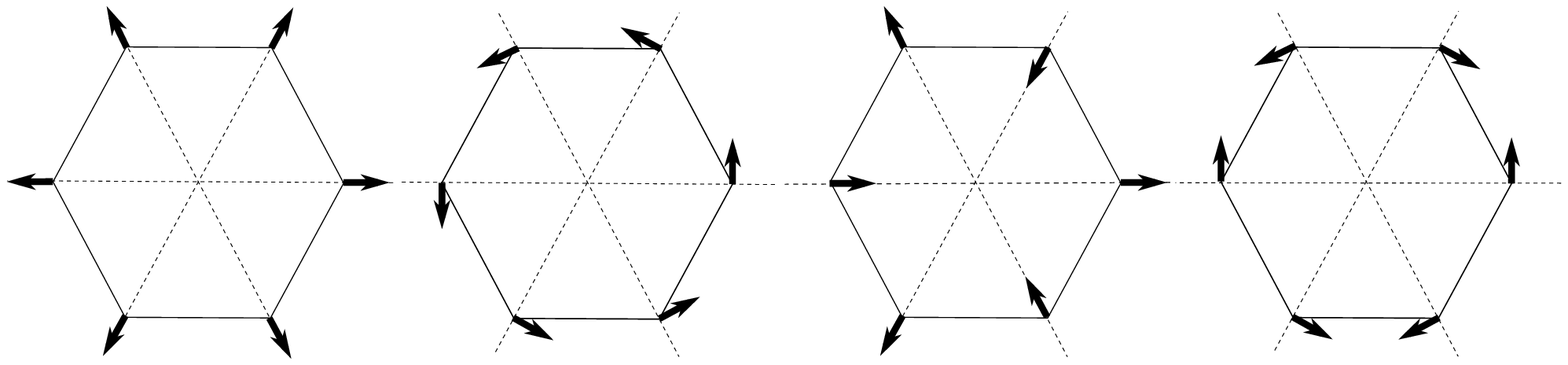}}
\caption{Basis displacements for $\vs^{(\tau)},\vs^{(\alpha)},\vs^{(\phi)},\vs^{(\psi)}$, respectively.}
\label{f0}
\end{center}
\end{figure}

 We briefly explain the construction of the displacements in figure~\ref{f0}. Let $x$ be the upper left vertex of the first hexagon. If $O$ is the center of the hexagon, consider the displacement $\delta_x$ which is the vector $n(x-O)$ at $x$ and the zero vector at the remaining vertices. Apply the projection $p^{(\tau)}$ to $\delta_x$. From equations~\eqref{ptau} and~\eqref{cyclic}, we have that
 \[ p^{(\tau)}(\delta_x)=c_0(r)(e+s)(\delta_x),\quad \text{where} \ \ c_0(r)=\frac{1}{2n}\sum_{j=1}^n r^j,\]
and $s$ is the reflection through the line generated by $x-O$. We calculate firstly $(e+s)(\delta_x)=2\delta_x$. Next we observe that applying $c_0(r)$ results in the sum of all the rotated images of $2 \delta_x$ and divided by $2n$. Thus we obtain a displacement which at vertex $r^jx$ is given by $r^jx-O$, $j=1,\cdots,n$. We find it convenient to denote such displacement by $\varkappa$, the symbol used to represent the configuration formed by the vertices of the $n$-gon. By applying the operators $p^{(\alpha)}J,p^{(\phi)}J$ and $p^{(\psi)}$ to $\delta_x$, keeping proposition~\ref{comJ} in mind, we obtain the last three displacements in figure~\ref{f0}.

 Let us now consider the isotypic components $\vs^{(k)}=\vs_1^{(k)}\oplus\vs_2^{(k)}$, $k=1,\cdots,t$, where henceforth $t=\frac{n}{2}-1$ if $n$ is even, or $t=\frac{n-1}{2}$ if $n$ is odd. Pick again one of the vertices of the $n$-gon, call it $x$, and let $s$ be the reflection through the line determined by $x$ and the center $O$ of the $n$-gon. Recall $\vs_i^{(k)}=\text{Image}(p_{ii}^{(k)})$, $i=1,2$, and let $\delta_x$ be the displacement defined in the previous paragraph. We compute the displacements
\[ \delta_1^{(k)}=p_{11}^{(k)}(\delta_x)=2 c_k(r)(e+s)(\delta_x), \quad \delta_2^{(k)}=p_{12}^{(k)}(J\delta_x)=2 s_k(r)(-e+s)(J\delta_x),\]
using the procedure explained in the preceding paragraph. Using that $p_{11}^{(k)}$ is an $M$-orthogonal projection with respect to the mass inner product~\eqref{inner} and identities~\eqref{algebra}, we obtain
\[ \langle \delta_1^{(k)},\delta_2^{(k)}\rangle_M=\langle \delta_x, \left(p_{11}^{(k)}\circ p_{12}^{(k)}\right)(J\delta_x)\rangle_M= \langle \delta_x,p_{12}^{(k)}(J\delta_x)\rangle_M=0. \]
Thus $\{\delta_1^{(k)},\delta_2^{(k)}\}$ is an $M$-orthogonal basis of $\vs_1^{(k)}$. We apply the transfer isomorphism $p_{21}^{(k)}$ in order to obtain the $M$-orthogonal basis of $\vs_2^{(k)}$ formed by the displacements
\[ \epsilon_1^{(k)}=p_{21}^{(k)}(\delta_1^{(k)})=p_{21}^{(k)}(\delta_x), \quad \epsilon_2^{(k)}=p_{21}^{(k)}(\delta_2^{(k)})=p_{22}^{(k)}(J\delta_x).\]
From remark~\ref{simantisim}, $A$ has a symmetric matrix representation with respect to any $M$-orthogonal basis. Let $a_1,a_2,a_3 \in \reals$ be such that
\begin{align*}
A\epsilon_2^{(k)}&=a_1 \epsilon_2^{(k)}+a_2 \epsilon_1^{(k)},\\
A \epsilon_1^{(k)}&= a_2 \epsilon_2^{(k)}+a_3 \epsilon_1^{(k)}.
\end{align*}
If we apply $p_{12}^{(k)}$ to both sides of the above equations, use~\eqref{algebra} and recall the definitions of $\delta_1^{(k)}, \delta_2^{(k)}$, we obtain
\begin{align*}
A \delta_1^{(k)}&= a_3 \delta_1^{(k)}+a_2 \delta_2^{(k)},\\
A\delta_2^{(k)}&=a_2 \delta_1^{(k)}+a_1 \delta_2^{(k)}.
\end{align*}
 Next we consider $J$. From the identities~\eqref{transJ} in proposition~\ref{comJ}, we have that
\begin{align*}
J\epsilon_2^{(k)}&=Jp_{22}^{(k)}(J\delta_x)=-p_{11}^{(k)}(\delta_x)=-\delta_1^{(k)},\\
J\epsilon_1^{(k)}&=Jp_{21}^{(k)}(\delta_x)=-p_{12}^{(k)}(J\delta_x)=-\delta_2^{(k)},\\
J\delta_1^{(k)}&=Jp_{11}^{(k)}(\delta_x)=p_{22}^{(k)}(J\delta_x)=\epsilon_2^{(k)},\\
J\delta_2^{(k)}&=Jp_{12}^{(k)}(J\delta_x)=p_{21}^{(k)}(\delta_x)=\epsilon_1^{(k)}.
\end{align*}
Therefore, with respect to the $M$-orthogonal basis $\{\epsilon_2^{(k)},\epsilon_1^{(k)},\delta_1^{(k)},\delta_2^{(k)}\}$, the matrix of $A+\kappa J$ has the form
\[ \left[\begin{array}{cccc}
          a_1 &a_2& \kappa & \\
          a_2 & a_3 & &\ \kappa\\
          -\kappa& & a_3 & a_2\\
           & -\kappa & a_2 & a_1
          \end{array}
   \right]
\]
This matrix is of the form found in~\cite{m95}, section 2, proposition 4. It is noteworthy that the block diagonalizations and related factorizations of the stability polynomial found in the literature can be explained, albeit not improved, by systematic usage of representation theory\footnote{In the case of~\cite{m95}, a close comparison between the expressions of our projections and transfer isomorphisms and the expression for $u \in \mathbb{C}^{2n}$ in lemma 1 of section 3 provides a reasonable justification.}. This is the case even for the pioneering work by J. C. Maxwell on the stability of the centered regular $n$-gon relative equilibria, cf. his essay~\cite{ma56}.

\subsubsection{Analysis of ${\vs^{(\sigma)}}$} \label{transdecI} It is possible to use translational symmetry to decompose $\vs^{(1)}=\vs^{(\sigma)}$ into invariant (and irreducible) subspaces so that $A+\kappa J$ assumes its simplest form. We observe that if $\delta_h$ and $\delta_v$ represent translations of $X$ in the horizontal and the vertical directions, respectively, then $\{\delta_h,\delta_v\}$ is a basis for a subrepresentation of $\sigma^E$ isomorphic to $\sigma$. Indeed, let $e_1=(1,0)$ and $e_2=(0,1)$  be the vectors in the canonical basis of $\reals^2$. Using the natural isomorphism between the fibers of $E$ and $\reals^2$, view $e_1$ and $e_2$ at each vertex of the $n$-gon $X$. Define the horizontal and vertical displacements of $X$ by
\[ \delta_h(x)=e_1, \quad \delta_v(x)=e_2, \qquad \forall x \in X.\]
From 
\[ \sigma^E(r^js^k)(\delta_h)(x)=r^js^ke_1, \quad \sigma^E(r^js^k)(\delta_v)(x)=r^js^ke_2, \]
for $j=1,\cdots,n, \ k=0,1, \ \forall x \in X $, we see that $\sigma^E(g)$ takes $\delta_h$ and $\delta_v$ to translations in different directions, for all $g \in D_n$. Clearly all such translations are linear combinations of $\delta_h$ and $\delta_v$, so $\delta_h$ and $\delta_v$ generate a $\sigma^E$-invariant subspace $\mathcal{T} \leq \sections$. From
\[ \sigma^E(r)(\delta_h)=r\delta_h, \quad \sigma^E(r)(\delta_v)=r \delta_v, \qquad \sigma^E(s)(\delta_h)=s\delta_h, \quad \sigma^E(s)(\delta_v)=s \delta_v,\]
it follows that $\sigma^E$ restricted to $\mathcal{T}=\text{span}\{\delta_h,\delta_v\}$ is given by $\sigma$. Since $\vs^{(\sigma)}$ contains all subrepresentations isomorphic to $\sigma$ (see theorem~\ref{maintheo}, (1)),  we have that $\mathcal{T} \leq \vs^{(\sigma)}$. We know $\vs^{(\sigma)}$ contains another copy of $\sigma$, and in order to produce it, we will determine $\mathcal{T}^{\perp_M}\leq \vs^{(\sigma)}$, the $M$-orthogonal complement of $\mathcal{T}$ within $\vs^{(\sigma)}$. Our first step is to find suitable expressions for $\delta_h$ and $\delta_v$.

\begin{prop} \label{transla}
For $i=1,2$, and $x \in X$ fixed, let
\[ \delta_{i,x}(\tilde{x})=\begin{cases}
                     e_i,& \text{if } \tilde{x}=x,\\
                     0,& \text{if } \tilde{x} \neq x.
                    \end{cases}
\]   
The horizontal and vertical displacements $\delta_h$ and $\delta_v$ are given by
\[ \delta_h=\frac{n}{2}\left[p_{11}^{(\sigma)}(\delta_{1,x})+p_{12}^{(\sigma)}(\delta_{2,x})\right], \qquad \delta_v=J \delta_h=\frac{n}{2}\left[p_{21}^{(\sigma)}(\delta_{1,x})+p_{22}^{(\sigma)}(\delta_{2,x})\right].
\]  
In particular, we have that $\delta_h \in \vs^{(\sigma)}_1$  and $\delta_v \in \vs_2^{(\sigma)}$.          
\end{prop}
\begin{proof}
 The reflection $s$ can be taken with respect to the line connecting the center $O$ of the $n$-gon to the vertex $x$. Without loss of generality, we may suppose $e_1$ is in the direction of that line. Thus $s\delta_{1,x}=\delta_{1,x}$ and $s\delta_{2,x}=-\delta_{2,x}$. Since $p_{11}^{(\sigma)}=p_{11}^{(1)}$ and $p_{12}^{(\sigma)}=p_{12}^{(1)}$, we have that
\[ p_{11}^{(\sigma)}=\frac{1}{n}\left(\sum_{j=1}^{n}\cos\frac{2 \pi j}{n}r^j\right)(e+s), \quad p_{12}^{(\sigma)}=\frac{1}{n}\left(\sum_{j=1}^{n}\sin\frac{2 \pi j}{n}r^j\right)(-e+s).\]
Hence we obtain, for each $j=1,\cdots,n$,
\[p_{11}^{(\sigma)}(\delta_{1,x})(r^jx)=\frac{2}{n}\left(\cos^2\frac{2 \pi j}{n}e_1+\cos \frac{2 \pi j}{n} \sin \frac{2 \pi j}{n} e_2\right),\]
and
\[ p_{12}^{(\sigma)}(\delta_{2,x})(r^jx)=\frac{2}{n}\left(\sin^2\frac{2 \pi j}{n}e_1-\cos \frac{2 \pi j}{n} \sin \frac{2 \pi j}{n} e_2\right). \]
After adding the two formulas above and multiplying the result by $n/2$, the expression for $\delta_h$ follows. 

 Since $\delta_v=J\delta_h$, using the identities~\eqref{transJ}, we deduce the expression for $\delta_v$ in the statement of the lemma. As long as $p_{11}^{(\sigma)}(\delta_h)=\delta_h$ and $\delta_v=p_{21}^{(\sigma)}(\delta_h)$, we have that $\delta_h \in \vs_1^{(\sigma)}$ and $\delta_v \in \vs_2^{(\sigma)}$. 
\end{proof}

\begin{prop} \label{translcomp} For $x \in X$ fixed, and $\delta_{1,x},\delta_{2,x}\in \sections$ as in proposition~\ref{transla}, let
\[ \epsilon_h=\frac{n}{2}\left[p_{11}^{(\sigma)}(\delta_{1,x})-p_{12}^{(\sigma)}(\delta_{2,x})\right], \qquad \epsilon_v= J \epsilon_h=\frac{n}{2}\left[p_{22}^{(\sigma)}(\delta_{2,x})-p_{21}^{(\sigma)}(\delta_{1,x})\right].\]
Then $\{\delta_h,\epsilon_h\}$ and $\{\delta_v,\epsilon_v\}$ are $M$-orthogonal bases of $\vs_1^{(\sigma)}$ and $\vs_2^{(\sigma)}$, respectively. In particular, $\epsilon_h$ and $\epsilon_v$ are eigenvectors of $A$.
\end{prop}
\begin{proof} From proposition~\ref{transla} we know that $\delta_h \in \vs_1^{(\sigma)}$ and $\delta_v \in \vs_2^{(\sigma)}$. By applying the projections $p_{11}^{(\sigma)}$ and $p_{22}^{(\sigma)}$ to $\epsilon_h$ and $\epsilon_v$, respectively, we conclude that $\epsilon_h \in \vs_1^{(\sigma)}$ and $\epsilon_v \in \vs_2^{(\sigma)}$. Notice that 
\[ \langle p_{11}^{(\sigma)}(\delta_{1,x}),p_{12}^{(\sigma)}(\delta_{2,x})\rangle_M=\langle \delta_{1,x},p_{12}^{(\sigma)}(\delta_{2,x})\rangle_M=0,\]
and 
\[ \langle p_{11}^{(\sigma)}(\delta_{1,x}),p_{11}^{(\sigma)}(\delta_{1,x})\rangle_M=\frac{2}{n}=\langle p_{12}^{(\sigma)}(\delta_{2,x}),p_{12}^{(\sigma)}(\delta_{2,x})\rangle_M.\]
Thus $\delta_h$ is $M$-orthogonal to $\epsilon_h$. Through similar calculations, we verify that $\delta_v$ is $M$-orthogonal to $\epsilon_v$. The proof is complete.
\end{proof}

 The form of the restriction of $A+\kappa J$ to $\vs^{(\sigma)}$ can now be determined. Consider the basis $\{\delta_v,\delta_h,\epsilon_v,\epsilon_h\}$ from propositions~\ref{transla} and~\ref{translcomp}. It is not hard to verify from direct calculations (starting from the invariance of the potential functions by translations) that the eigenvalues associated with $\delta_h$ and $\delta_v$ are equal to 0. Furthermore, $\{\epsilon_h,\epsilon_v\}$ is a basis of $\mathcal{T}^{\perp_M}$ formed by eigenvectors of $A$ according to proposition~\ref{translcomp}, and the restriction of $\sigma^E$ to $\mathcal{T}^{\perp_M}$ is irreducible, thus isomorphic to $\sigma$.
It follows that $\mathcal{T}^{\perp_M}$ is contained in a single eigenspace of $A$. Let $\widetilde{\lambda}$ be the eigenvalue of $A$ associated with  $\epsilon_h$ and $\epsilon_v$. Therefore the form of $A+\kappa J$ with respect to the basis $\{\delta_v,\delta_h,\epsilon_v,\epsilon_h\}$ is
\[ \left[\begin{array}{cccc}
          0 &\kappa&  & \\
          -\kappa & 0 & &\\
          & & \widetilde{\lambda} & \kappa\\
           &  & -\kappa & \widetilde{\lambda}
          \end{array}
   \right]
\]
\begin{rem} \label{AcommJ} The above results showed that, for $n$-gons, besides the eigenvectors  traditionally associated with the symmetries and homogeneity of the potential functions, there are additional eigenvectors of $A$ in $\vs^{(\phi)},\vs^{(\psi)}$ and $\vs^{(\sigma)}$. As a matter of fact, all eigenvectors of $A$ could be determined from its block-diagonal form. However, the key property that the image by $J$ of an eigenvector is also an eigenvector need not hold in general, so the matrix forms of operators $A+\kappa J$ with respect to a basis of eigenvectors of $A$ may not be so simple. Interestingly, however, as noted by Roberts~\cite{gr13}, in the $N$-vortex problem the image by $J$ of an eigenvector is always an eigenvector. So, for the $n$-gon relative equilibrium with equal vorticities, a full simplification of $A+\kappa J$ is achievable. In section~\ref{sec10}, we show that full factorizations can also be achieved for some $4$-body $D_2$-symmetric ring systems in the $N$-body and $N$-vortex problems.
\end{rem}

\subsection{Block diagonalization for semiregular $\bm{2n}$-gons} \label{sec8} For a semiregular $2n$-gon we have that $\sigma^E \cong 2\regd$, thus
\[ \dim \vs^{(\tau)}=\dim \vs^{(\alpha)}=\dim \vs^{(\phi)}=\dim \vs^{(\psi)}=2, \ \  \dim \vs^{(k)}=8,  \ \ \forall k=1,\cdots,t. \]

 In order to find a basis for $\vs^{(\tau)}$, we apply the projection $p^{(\tau)}$ to two $M$-orthogonal displacements. Let $x$ be a fixed vertex, $O$ be the center of the $2n$-gon, and let $\delta_x$ be the displacement given by the vector $2n(x-O)$ at $x$ and by the null vector at the remaining vertices. We assume $s$ is the reflection about the perpendicular bisector of a side of the $2n$-gon having $x$ as one of its endpoints. The displacements
\[ p^{(\tau)}(\delta_x)=\varkappa, \quad p^{(\tau)}(J\delta_x)=\varkappa'\]
form an $M$-orthogonal basis of the isotypic component $\vs^{(\tau)}$. From remarks~\ref{simantisim} and~\ref{nbody-nvortex}, since $\varkappa$ is an eigenvector of $A$, so is $\varkappa'$, but the respective eigenvalues need not be the same.

 Using the identities~\eqref{projsJ}, we obtain the $M$-orthogonal basis for $\vs^{(\alpha)}$ consisting of the displacements $J\varkappa$ and $J \varkappa'$.
Thus $\vs^{(\tau)}\oplus \vs^{(\alpha)}$ has an $M$-orthogonal basis $\{J\varkappa,\varkappa,J\varkappa',\varkappa'\}$ whose elements are eigenvectors of $A$, and with respect to which the operators $A+\kappa J$ have matrices of the form
\[ \left[\begin{array}{cccc}
          * &\kappa&  & \\
          -\kappa & * & &\\
          & & * & \kappa\\
           &  & -\kappa & *
          \end{array}
   \right].
\]
 
  In an analogous manner, we find an $M$-orthogonal basis for $\vs^{(\phi)}$, namely $\{p^{(\phi)}(\delta_x),$ $p^{(\phi)}(J \delta_x)\}$, and an $M$-orthogonal basis for $\vs^{(\psi)}$, viz. $\{p^{(\psi)}(J\delta_x),-p^{(\psi)}(\delta_x)\}$. The form of $A+\kappa J$ with respect to the basis $\{p^{(\psi)}(J\delta_x),-p^{(\psi)}(\delta_x),p^{(\phi)}(\delta_x),p^{(\phi)}(J \delta_x)\}$ is
\[ \left[\begin{array}{cccc}
          a_1 &a_2& \kappa & \\
          a_2 & a_3 & &\ \kappa\\
          -\kappa& & a_1' & a_2'\\
           & -\kappa & a_2' & a_3'
          \end{array}
   \right].
\]
Figure~\ref{fx} depicts some of the basic displacements. Notice that the primed and unprimed entries of the above matrix are unrelated due to the absence of a transfer isomorphism\footnote{Exceptionally, though, in the $N$-vortex problem, $J$ plays the role of a transfer isomorphism -- see remark~\ref{AcommJ}.}.

\begin{figure}[h]
\begin{center}
\scalebox{.6}{\includegraphics{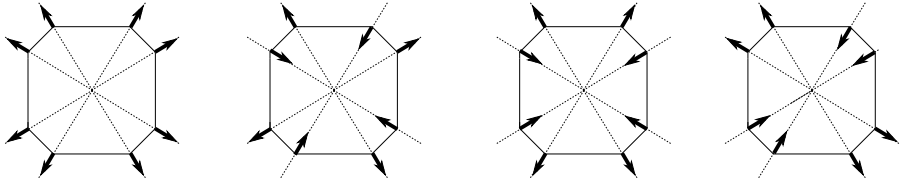}}
\caption{The images (up to scaling) of the displacement $\delta_x$ by the projections $p^{(\tau)}$, $p^{(\alpha)}$, $p^{(\phi)}$ and $p^{(\psi)}$.}
\label{fx}
\end{center}
\end{figure}

 Let us consider the isotypic components $\vs^{(k)}=\vs^{(k)}_1\oplus \vs^{(k)}_2$, $k=1,\cdots,t$. Let $\delta_{i,x}$ be the displacements defined in proposition~\ref{transla}. The four displacements
\[ p_{11}^{(k)}(\delta_{i,x}), \  p_{12}^{(k)}(\delta_{i,x}),\qquad i=1,2\] 
are pairwise $M$-orthogonal, as we can verify using proposition~\ref{standardMprojtrans}. Indeed, we have that
\[ \langle p_{11}^{(k)}(\delta_{i,x}),p_{12}^{(k)}(\delta_{i,x})\rangle_M=\langle \delta_{i,x},p_{12}^{(k)}(\delta_{i,x})\rangle_M=0, \qquad i=1,2,\]
and
\[ \langle p_{12}^{(k)}(\delta_{1,x}),p_{12}^{(k)}(\delta_{2,x})\rangle_M=\langle p_{11}^{(k)}(\delta_{1,x}),p_{11}^{(k)}(\delta_{2,x})\rangle_M=0.\]
Thus $\{p_{11}^{(k)}(\delta_{1,x}),p_{12}^{(k)}(\delta_{2,x}), p_{11}^{(k)}(\delta_{2,x}),p_{12}^{(k)}(\delta_{1,x})\}$ is an $M$-orthogonal basis of $\vs_1^{(k)}$. An $M$-orthogonal basis for $\vs_2^{(k)}$ can be produced by applying $J$ to each element of the basis of $\vs_1^{(k)}$:
\begin{align*}
Jp_{11}^{(k)}(\delta_{1,x})&=p_{22}^{(k)}(\delta_{2,x}), \quad Jp_{11}^{(k)}(\delta_{2,x})=-p_{22}^{(k)}(\delta_{1,x}),\\
Jp_{12}^{(k)}(\delta_{2,x})&=p_{21}^{(k)}(\delta_{1,x}), \quad Jp_{12}^{(k)}(\delta_{1,x})=-p_{21}^{(k)}(\delta_{2,x}),
\end{align*}
where we have used identities~\eqref{transJ} and that $\delta_{2,x}=J\delta_{1,x}$. 

 In order to determine the matrix forms of operators $A+\kappa J$, we must apply $p_{21}^{(k)}$ to the elements of the basis of $\vs_1^{(k)}$. The resulting displacements are, respectively: $p_{21}^{(k)}(\delta_{1,x}),p_{22}^{(k)}(\delta_{2,x}),p_{21}^{(k)}(\delta_{2,x})$ and $p_{22}^{(k)}(\delta_{1,x})$.
Therefore the matrices of operators $A+\kappa J$ with respect to the basis
\begin{align*}
 \{p_{22}^{(k)}(\delta_{2,x}),p_{21}^{(k)}(\delta_{1,x}), p_{22}^{(k)}(\delta_{1,x}),-p_{21}^{(k)}(\delta_{2,x}),\\
  p_{11}^{(k)}(\delta_{1,x}),p_{12}^{(k)}(\delta_{2,x}), -p_{11}^{(k)}(\delta_{2,x}),p_{12}^{(k)}(\delta_{1,x})\}
\end{align*}
have the form
\[ \left[\begin{array}{cccccccc}
         a_1&a_2&a_3&a_4&\kappa & & &\\
         a_2 & a_5 & a_6 & a_7 & & \kappa & & \\
         a_3 & a_6 & a_{8} & a_{9} & & & \kappa & \\
         a_4 & a_7 & a_{9} & a_{10}& & & &\kappa\\
         -\kappa & & & & a_5 & a_2 & a_7 & a_6\\
         &-\kappa & & & a_2 & a_1 & a_4 &a_3\\
         & & -\kappa & & a_7&a_4 & a_{10}&a_9\\
         & & & -\kappa & a_6 & a_3&a_9&a_8
         \end{array}
   \right]
\]

\subsubsection{Analysis of ${\vs^{(\sigma)}}$} \label{transdecII} As in the case of a regular $n$-gon, the isotypic component $\vs^{(\sigma)}=\vs_1^{(\sigma)}\oplus \vs_2^{(\sigma)}$ for a semiregular $2n$-gon $X$ can be given a finer decomposition $\mathcal{T}\oplus \mathcal{T}^{\perp_M}$, where $\mathcal{T}$ is the subspace formed by the translations of $X$.  It can be verified that the restrictions of $\sigma^E$ to $\mathcal{T}$ and $\mathcal{T}^{\perp_M}$ are isomorphic to $\sigma$ and $3 \sigma$, respectively. We will show that $\mathcal{T}\cap \vs_i^{(\sigma)}\neq \{0\}$, $i=1,2$, so it makes sense to find suitable $M$-orthogonal complements within $\vs_1^{(\sigma)}$ and $\vs_2^{(\sigma)}$.

 Let $x$ and $x'=sx$ be neighboring vertices; $s$ is the reflection through the perpendicular bisector of the side $xx'$. Consider the displacements $\delta_{i,x}$, $\delta_{i,x'}$, $i=1,2$ as defined in proposition~\ref{transla}, and assume $i=1$ corresponds to the direction of the line fixed by $s$. Hence $s\delta_{1,x}=\delta_{1,x'}$ and $s\delta_{2,x}=-\delta_{2,x'}$. Calculations similar to the ones performed in the proof of proposition~\ref{transla} lead us to the formulas in the next proposition.
 
\begin{prop} \label{transsemi} The horizontal and vertical translations of a semiregular $2n$-gon are respectively generated by the following displacements:
\[ \delta_h=n\left[p_{11}^{(\sigma)}(\delta_{1,x})+p_{12}^{(\sigma)}(\delta_{2,x})\right], \qquad \delta_v=n\left[p_{21}^{(\sigma)}(\delta_{1,x})+p_{22}^{(\sigma)}(\delta_{2,x})\right].
\]
\end{prop}
In addition, as in proposition~\ref{translcomp}, let
\[ \epsilon_h=n\left[p_{11}^{(\sigma)}(\delta_{1,x})-p_{12}^{(\sigma)}(\delta_{2,x})\right], \qquad \epsilon_v=n\left[p_{22}^{(\sigma)}(\delta_{2,x})-p_{21}^{(\sigma)}(\delta_{1,x})\right].\]
We have that $\delta_h,\epsilon_h \in \vs_1^{(\sigma)}$, $\delta_v,\epsilon_v \in \vs_2^{(\sigma)}$, and
\[ \delta_v=J \delta_h =p_{21}^{(\sigma)}(\delta_h), \qquad \epsilon_v=J \epsilon_h =p_{21}^{(\sigma)}(\epsilon_h).\]

 In order to obtain $M$-orthogonal bases for $\vs_1^{(\sigma)}$ and $\vs_2^{(\sigma)}$, we consider the displacements $\delta_1=p_{11}^{(\sigma)}(\delta_{2,x})$ and $\epsilon_1=p_{12}^{(\sigma)}(\delta_{1,x})$, which belong to $\vs_1^{(\sigma)}$. Straightforward computations show that $\delta_1$ and $\epsilon_1$ are $M$-orthogonal, and that $\delta_1,\epsilon_1 \in \{\delta_h,\epsilon_h\}^{\perp_M} \leq \vs_1^{(\sigma)}$. For instance, we have that
\begin{align*}
 \langle \epsilon_1,\epsilon_h \rangle_M= \langle \delta_{1,x},p_{21}^{(\sigma)}(\epsilon_h)\rangle=\langle\delta_{1,x},\epsilon_v\rangle_M=0,
\end{align*}
since
\begin{align*}
\epsilon_v(x)&=n\left[p_{22}^{(\sigma)}(\delta_{2,x})(x)-p_{21}^{(\sigma)}(\delta_{1,x})(x)\right]\\
             &= n\left[ \big(2c_1(r)(e-s)(\delta_{2,x})\big)(x)-\big(2s_1(r)(e+s)(\delta_{1,x})\big)(x)\right]\\
             &= \delta_{2,x}.
\end{align*}
We conclude that $\{\delta_h,\epsilon_h,\delta_1,\epsilon_1\}$ is an $M$-orthogonal basis of $\vs_1^{(\sigma)}$. Define
\[ \delta_2=J \delta_1, \qquad \epsilon_2 = J\epsilon_1.\]
Then $\{\delta_v,\epsilon_v,\delta_2,\epsilon_2\}$ is an $M$-orthogonal basis of $\vs_2^{(\sigma)}$, and we have that
\[ p_{12}^{(\sigma)}(\epsilon_v)=-\epsilon_h, \quad p_{12}^{(\sigma)}(\delta_2)=-\epsilon_1, \quad \text{and} \quad p_{12}^{(\sigma)}(\epsilon_2)=-\delta_1 .\]

 Recall $\{\delta_h,\delta_v\} \leq \mathcal{T} \leq \text{Ker}(A)$. Together with the results of the previous paragraph, we deduce that the possible forms of operators $A+\kappa J$ with respect to the basis of $\vs^{(\sigma)}$ given by $\{\delta_v,\delta_h,\epsilon_v,\delta_2,\epsilon_2,\epsilon_h,\delta_1,\epsilon_1\}$ are
\[ \left[\begin{array}{cccccccc}
   0&\kappa& & & & & & \\
   -\kappa & 0 &&&&& \\
   &&a_1 & a_2 & a_3&\kappa & & \\
   &&a_2 & a_4 & a_5 & &\kappa & \\
   && a_3 & a_5 & a_6 & & & \kappa \\
   && -\kappa & & & -a_1 & -a_2 &- a_3 \\
   && & -\kappa && -a_2&-a_6 &-a_5 \\
   &&&& -\kappa &-a_3&-a_5&-a_4
   \end{array}
   \right].
\]

\section{Factorization of the Stability Polynomial of a Ring System with Multiple Rings} \label{sec9}

 As in section~\ref{1orbit}, we suppose $n>2$. See remark~\ref{D2esp}.
 
 Let $X$ be a ring system and $X=\cup_{i=1}^l \mathcal{O}_i$ be its $D_n$-orbit decomposition.
Recall proposition~\ref{rings} on the structure of ring systems. We say a ring system $X$ has type $(a,b,c)$ if $X$ has $a=1$ or $0$ points at its barycenter, and $X$ contains precisely $b$ regular $n$-gons and $c$ semiregular $2n$-gons. 
From the fact that $J$ and the projections and transfer isomorphisms associated with $\sigma^E$ respect the decomposition $\sections=\oplus_{i=1}^l\Gamma\left(T\reals^2\big|_{\mathcal{O}_i}\right)$, the block diagonalization of $A$ can be accomplished by using the natural embeddings of the space of displacements of each individual orbit, $\Gamma\left(T\reals^2\big|_{\mathcal{O}_i}\right)$, into $\sections$. We employ the results of section~\ref{1orbit} to find the suitable $M$-orthogonal bases for each $\Gamma\left(T\reals^2\big|_{\mathcal{O}_i}\right)$, apply the corresponding embedding and take the union of the images, thus producing an $M$-orthogonal basis of $\sections$ which block-diagonalizes operators of the form $A+\kappa J$. A simple reordering of the elements of the resulting basis of $\sections$ puts each block of $J$ in standard form. This procedure provides the block diagonalization based on the $D_n$-symmetry of $X$. In addition, by making use of the eigenvectors associated with the symmetries and homogeneity of the potential functions $U_{\gamma}$, and the symmetries of $H_{-1}$, refinements can be obtained for the blocks corresponding to the sum of isotypic components $\vs^{(\tau)}\oplus\vs^{(\alpha)}$, and for the blocks corresponding to the isotypic component $\vs^{(\sigma)}$. In the next two subsections we present a method for obtaining such refinements.

\subsection{Analysis of $\bm{\vs^{(\tau)}\oplus\vs^{(\alpha)}}$} \label{dilationrotation} Consider a ring system $X$. We express the displacement corresponding to $\varkappa$ through the equation
\[ \varkappa= p^{(\tau)}(\delta_{x_1})+\cdots+p^{(\tau)}(\delta_{x_l}),\]
where $x_i \in \mathcal{O}_i$, $i=1,\cdots,l$, and $\delta_{x_i}$ is the displacement of $X$ that is equal to the null displacement if $\mathcal{O}_i=\{O\}$, or equal to the vector $n(x_i-O)$ at $x_i$ and the zero vector at all $x \in X \setminus \{x_i\}$ if $\mathcal{O}_i$ is an $n$-gon, or equal to $2n(x_i-O)$ at $x_i$ and the zero vector at all $x \in X \setminus \{x_i\}$ if $\mathcal{O}_i$ is a semiregular $2n$-gon. 

 If the ring system is of type $(0,b,c)$, we may be able to determine an $M$-orthogonal basis for $\vs^{(\tau)}$ as follows. Pick one of the orbits and call it $\mathcal{O}_1$. For each $i\in \{2,\cdots,l\}$, let
\[ \varkappa_i=p^{(\tau)}(\delta_{x_1})+c_i p^{(\tau)}(\delta_{x_i}),\]
where each $c_i$ is a nonzero constant
chosen so that the orthogonality identities
\[ \langle \varkappa, \varkappa_i\rangle_M=0, \qquad i=2,\cdots,l\]
are verified. More explicitly, if $R_i$ is the radius of the circle circumscribing the orbit $\mathcal{O}_i$ and $m_i$ is the mass of the bodies in $\mathcal{O}_i$, we have that
\[ \langle p^{(\tau)}(\delta_{x_i}),p^{(\tau)}(\delta_{x_i})\rangle_M=m_iR_i^2 |\mathcal{O}_i|, \qquad i=1,\cdots,l, \]
hence
\[ c_i=-\frac{m_1 R_1^2 |\mathcal{O}_1|}{m_i R_i^2 |\mathcal{O}_i|}, \qquad i=2,\cdots,l.\]
Finally, for each semiregular $2n$-gon $\mathcal{O}_j$ in $X$, we add the displacement 
\[\varkappa_j'=p^{(\tau)}(J\delta_{x_j}).\]
It is clear that 
\[ \langle \varkappa_j',\varkappa_j'\rangle_M\neq 0, \ \ \langle \varkappa,\varkappa_j'\rangle_M=\langle \varkappa_i,\varkappa_j'\rangle_M=0, \qquad \forall i,j,\]
albeit the $\varkappa_i$ may neither be pairwise $M$-orthogonal nor have nonzero $M$-norm (if $c_i=-1$). In general, though, the displacements $\varkappa,\varkappa_i,\varkappa_j'$ do form a basis of $\vs^{(\tau)}$ which may be $M$-orthogonalized (via the Gram-Schmidt process) depending on the values of the masses on each orbit. In the special case of two regular $n$-gons, the basis $\{\varkappa,\varkappa_2\}$ of $\vs^{(\tau)}$ is $M$-orthogonal and, from remarks~\ref{simantisim} and~\ref{nbody-nvortex}, $\varkappa$ and $\varkappa_2$ are eigenvectors of $A$.

 As it was done previously, a basis for $\vs^{(\alpha)}$ can be obtained by simply applying $J$ to the basis of $\vs^{(\tau)}$. This basis includes the eigenvector $J \varkappa$. If we take the union of the bases of $\vs^{(\tau)}$ and  $\vs^{(\alpha)}$ so that the pair $J\varkappa,  \varkappa$ appears firstly, followed by  the sequence $J \varkappa_2, \cdots,J \varkappa_l,J\varkappa_1',\cdots,J\varkappa_c'$, and subsequently by $\varkappa_2,\cdots,\varkappa_l,\varkappa_1',\cdots,\varkappa_c'$, the matrices of $A+\kappa J$ will have a block-diagonal form consisting of one $2 \times 2$ block followed by one $2(l-1+c)\times 2(l-1+c)$ block. Within each of these two blocks, $A$ has a block-diagonal structure formed by two blocks of size $1\times 1$ and two blocks of size $(l-1+c)\times (l-1+c)$, respectively, and $J$ is in standard form. 
 
 If $X$ contains its barycenter $O$, i.e., $X$ is a ring system of type $(1,b,c)$, then all the above remains valid as long as we replace $X$ with $X\setminus\{O\}$ and $l$ with $l-1$. The isotypic components $\vs^{(\tau)}$ and $\vs^{(\alpha)}$ in $\Gamma\left(T\reals^2\big|_{X \setminus \{O\}}\right)$ embed naturally into $\sections$ since $\sigma^E\Big|_{\Gamma\left(T\reals^2\big|_{\{O\}}\right)}\simeq \sigma$.

\subsection{Analysis of $\bm{\vs^{(\sigma)}}$} \label{transl}  The displacements corresponding to horizontal and vertical translations of $X$ are respectively multiples of
\[ \Delta_{h}(x)=e_1,\quad \Delta_v(x)=e_2,\qquad \forall x \in X.\]
Let $\ell$ be the line fixed by the reflection $s$, and suppose $e_1$ is parallel to $\ell$. Pick one vertex $x_i$ from each regular $n$-gon in $X$, one vertex $x_j'$ from each semiregular $2n$-gon in $X$ and, using the notation of proposition~\ref{transla}, form the displacements\footnote{If $n$ is even, a more careful discussion should separate the regular $n$-gons in $X$ in two families: the ones for which the line $\ell$ goes through a vertex and those for which $\ell$ does not go through a vertex. A short calculation shows that translations of either family are described by the same formula.}
\[\Delta_k=a\delta_{k,O}+n\left(\sum_{i=1}^b \delta_{k,x_i}\right)+2n\left(\sum_{j=1}^c \delta_{k,x_j'}\right), \quad k=1,2.\]
From the formulas in propositions~\ref{transla} and~\ref{transsemi}, we have
\[ \Delta_h=\frac{1}{2}\left[p_{11}^{(\sigma)}\left(\Delta_1\right)+p_{12}^{(\sigma)}(\Delta_2)\right], \qquad \Delta_v=\frac{1}{2}\left[p_{21}^{(\sigma)}(\Delta_{1})+p_{22}^{(\sigma)}(\Delta_{2})\right]. \]
It is clear that $\Delta_h \in \vs_1^{(\sigma)}$, $\Delta_v \in \vs_2^{(\sigma)}$, and  $\Delta_v=J\Delta_h$.

 We wish to construct an $M$-orthogonal basis for $\vs^{(\sigma)}$ which contains $\Delta_h$ and $\Delta_v$. For each $i=1,\cdots,l$, let $\delta_{h}^{(i)}$ and $\delta_v^{(i)}$ denote the horizontal and vertical translations of the orbit $\mathcal{O}_i$, respectively, viewed as displacements of $X$. The following identities hold
\[ \Delta_h=\delta_h^{(1)}+\cdots+\delta_h^{(l)}, \qquad \Delta_v=\delta_v^{(1)}+\cdots+\delta_v^{(l)},\]
and each summand in either identity is $M$-orthogonal to all the remaining summands in both equations. Besides, if $m_i$ is the mass of the particles in the orbit $\mathcal{O}_i$, we have that
\[ \langle \delta_h^{(i)},\delta_h^{(i)}\rangle_M=\langle \delta_v^{(i)},\delta_v^{(i)}\rangle_M=m_i|\mathcal{O}_i|.\]
As in subsection~\ref{dilationrotation}, define the displacements
\begin{equation} \label{ihorver}
 \Delta_h^{(i)}= \delta_h^{(1)}+c_i \delta_h^{(i)}, \qquad \Delta_v^{(i)}= J \Delta_h^{(i)}, \qquad i=2,\cdots,l,
\end{equation}
where the $c_i$ are constants selected in order to validate the orthogonality conditions:
\[ \langle \Delta_h, \Delta_h^{(i)}\rangle_M=0, \qquad i=2,\cdots,l. \]

 In order to obtain a suitable basis for $\vs^{(\sigma)}$, we consider each orbit $\mathcal{O}_i$ in $X$ separately. If $\mathcal{O}_i$ is the barycenter of $X$, we form the $M$-orthogonal set of displacements:
\[  \mathfrak{T}_i=\{\Delta_v^{(i)},\Delta_h^{(i)}\}, \]
and if $\mathcal{O}_i$ is a regular $n$-gon, we form the $M$-orthogonal set
\[ \mathfrak{T}_i=\{\Delta_v^{(i)},\epsilon_v^{(i)},\Delta_h^{(i)},\epsilon_h^{(i)}\},\]
where $\epsilon_h^{(i)},\epsilon_v^{(i)}$ are the natural embeddings into $\sections$ of the displacements $\epsilon_h,\epsilon_v$ from proposition~\ref{translcomp}. Finally, if $\mathcal{O}_i$ is a semiregular $2n$-gon, then we form the $M$-orthogonal set
\[ \mathfrak{T}_i=\{\Delta_v^{(i)},\epsilon_v^{(i)},\delta_2^{(i)},\epsilon_2^{(i)},\Delta_h^{(i)},\epsilon_h^{(i)},\delta_1^{(i)},\epsilon_1^{(i)} \}, \]
where $\epsilon_h^{(i)},\delta_1^{(i)},\epsilon_1^{(i)},\epsilon_v^{(i)},\delta_2^{(i)},\epsilon_2^{(i)}$ are the natural embeddings of the displacements $\epsilon_h,\delta_1,\epsilon_1,\epsilon_v,\delta_2,\epsilon_2$ from~\ref{transdecII}. Next, for $i=1$, we replace $\Delta_h^{(i)},\Delta_v^{(i)}$ with $\Delta_h, \Delta_v$, and if $i>1$, we use the displacements in~\eqref{ihorver}. Taking the union of the sets $\mathfrak{T}_i$, $i=1,\cdots,l$, and reordering its elements conveniently, we obtain a basis for $\vs^{(\sigma)}$ such that operators $A+\kappa J$ have a block-diagonal form with a $2 \times 2$ block (with zeros on the main diagonal) followed by a $2(a+2b+4c-1)\times 2(a+2b+4c-1)$ block. Within each of these two blocks, $A$ has a block-diagonal structure formed by two blocks of size $1\times 1$ and two blocks of size $(a+2b+4c-1)\times (a+2b+4c-1)$, respectively, and $J$ is in standard form.

\begin{rem} \label{g-s}
 The basis for $\vs^{(\sigma)}$ constructed above is not $M$-orthogonal. In order to obtain an $M$-orthogonal basis (if mass values allow) we apply the Gram-Schmidt process to the displacements $\Delta_h^{(2)},\cdots,\Delta_h^{(l)}$ and collect the resulting displacements $\widetilde{\Delta}_h^{(2)},\cdots,\widetilde{\Delta}_h^{(l)}$. Notice all tilded displacements are still in $\vs_1^{(\sigma)}$. We then apply $J$ to each $\widetilde{\Delta}_h^{(i)}$ in order to obtain the corresponding $M$-orthogonal displacements $\widetilde{\Delta}_v^{(2)},\cdots,\widetilde{\Delta}_v^{(l)}$ in $\vs_2^{(\sigma)}$.
\end{rem} 

\begin{exam} Suppose $X \subset \reals^2$ is $D_n$-symmetric and consists of its barycenter $O$ with mass $m_1$ and two regular $n$-gons whose vertices have masses $m_2$ and $m_3$, respectively. Let $\mathcal{O}_1=\{O\}$. We have that
\[ \langle \delta_h^{(1)},\delta_h^{(1)}\rangle_M=m_1, \quad \langle \delta_h^{(2)},\delta_h^{(2)}\rangle_M=n m_2, \quad \langle \delta_h^{(3)},\delta_h^{(3)} \rangle_M=n m_3,\]
and $c_2=-\frac{m_1}{n m_2}$, $c_3=-\frac{m_1}{n m_3}$. Since $\sigma^E\big|_{\vs^{(\sigma)}}\simeq 5 \sigma$, the basis of $\vs^{(\sigma)}$  constructed above must have ten elements. This basis is formed by the displacements
\[ \Delta_v,\Delta_h,\Delta_v^{(2)},\epsilon_v^{(2)},\Delta_h^{(2)},\epsilon_h^{(2)}, \Delta_v^{(3)},\epsilon_v^{(3)},\Delta_h^{(3)},\epsilon_h^{(3)}, \]
where, for $x \in \mathcal{O}_i$, we have $\epsilon_h^{(i)}(r^jx)=\cos\frac{4 \pi j}{n}e_1+\sin \frac{4 \pi j}{n} e_2$, and $\epsilon_v^{(i)}=J\epsilon_h^{(i)}$, $i=2,3$. We apply Gram-Schmidt to $\Delta_h^{(2)},\Delta_h^{(3)}$, thereby obtaining an $M$-orthogonal pair $\widetilde{\Delta}_h^{(2)},\widetilde{\Delta}_h^{(3)}$. Let $\widetilde{\Delta}_v^{(i)}=J\widetilde{\Delta}_h^{(i)}$, $i=2,3$.  Thus we encounter an $M$-orthogonal basis for $\vs^{(\sigma)}$, namely
\[ \Delta_v,\Delta_h,\widetilde{\Delta}_v^{(2)},\epsilon_v^{(2)},\widetilde{\Delta}_h^{(2)},\epsilon_h^{(2)}, \widetilde{\Delta}_v^{(3)},\epsilon_v^{(3)},\widetilde{\Delta}_h^{(3)},\epsilon_h^{(3)}. \]
Figure~\ref{f2} illustrates the case $n=4$ with masses $m_1=4$, $m_2=1/2$ and $m_3=1$.

\begin{figure}[h]
\begin{center}
\scalebox{.6}{\includegraphics{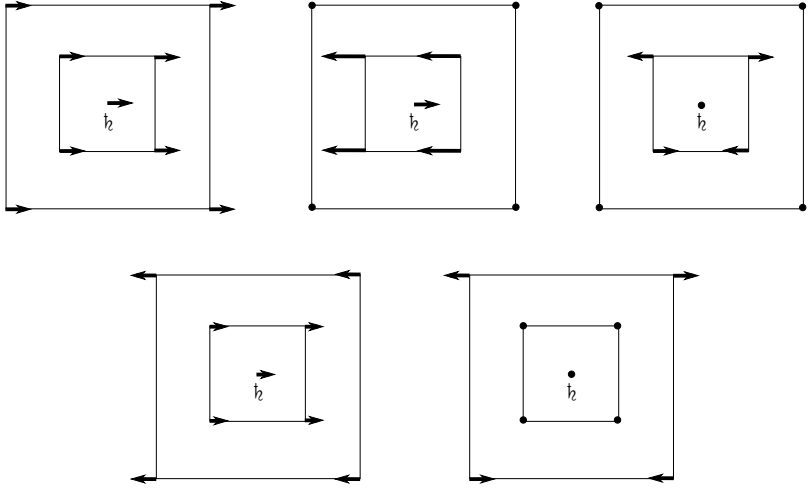}}
\caption{The displacements $\Delta_h,\widetilde{\Delta}_h^{(2)},\epsilon_h^{(2)},\widetilde{\Delta}_h^{(3)},\epsilon_h^{(3)}$ forming an $M$-orthogonal basis of $\vs_1^{(\sigma)}$ for a ring system consisting of a point $\saturn(=O)$ and the vertices of two homothetic squares.}
\label{f2}
\end{center}
\end{figure}

\end{exam}

\section{Full Factorization of the Stability Polynomial of Some $D_2$-Symmetric Sets} \label{sec10}

 In this section we show that the simplest block-diagonal form for operators $A+\kappa J$ can be achieved for $D_2$-symmetric sets with four elements. The key role is played by the refinements of the isotypic decompositions obtained from the usage of the classical symmetries (translations and rotation) and scaling. As a particular consequence, the stability polynomial of the rhombus in the $N$-body and $N$-vortex problems can always be fully factored.
 
 Consider the dihedral group $D_2=\{e,r,s,rs\}$. The table of irreducible representations of $D_2$ is given below. 
\vspace{.2cm}
\begin{center}
\begin{tabular}{c|cccc} \label{D2}
 &  $e$ & $r$ & $s$ & $rs$\\
\hline
$\tau$ & $1$ & 1& $1$ & 1 \\
$\alpha$ & $1$ &1 & $-1$ &$-1$\\
$\phi$ & $1$ & $-1$&  1 &$-1$\\
$\psi$ & 1 & $-1$ & $-1$ & 1
\end{tabular}
\end{center}
\vspace{.2cm}
If $c(g)=\frac{1}{2}(e+g)$, $g \in D_2$, the projections associated with the irreducible representations are:
\[
p^{(\tau)}=c(r)c(s), \quad
p^{(\alpha)} =c(r)c(-s), \quad
p^{(\phi)} =c(-r)c(s), \quad 
p^{(\psi)} =c(-r)c(-s).
\]

 Any $X \subset \reals^2$ with two elements can be viewed as a $D_2$-symmetric set. We view $r$ as a rotation of $\pi$ around the midpoint of $X$, and $s$ as a reflection about the line determined by $X$. The canonical representation $\sigma^E$ on the space of displacements $\sections$ of $X$ is isomorphic to $\rho^{\text{reg}}_{D_2}$, so, according to remark~\ref{regrepequiv}, we have the isotypic decomposition
\[ \sections = \vs^{(\tau)} \oplus \vs^{(\alpha)} \oplus \vs^{(\phi)} \oplus \vs^{(\psi)},\]
where each of the summands on the right-hand side is a one-dimensional subspace. Since each isotypic component is $A$-invariant, each such component is generated by an eigenvector of $A$. The displacements corresponding to the classical symmetries, namely $\varkappa, J \varkappa, \delta_h, \delta_v$, are generators of the isotypic components and thus form a basis of $\sections$. Figure~\ref{f5} illustrates generators of $\vs^{(\tau)}$ and $\vs^{(\phi)}$. As before, generators for $\vs^{(\alpha)}$ and $\vs^{(\psi)}$ can be obtained through the application of $J$.
\begin{figure}[h]
\begin{center}
\scalebox{.5}{\includegraphics{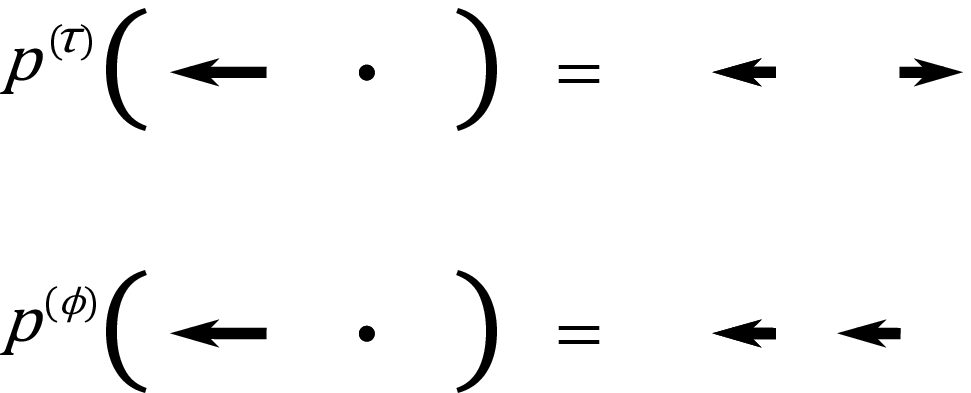}}
\caption{Generators of $\vs^{(\tau)}$ and $\vs^{(\phi)}$ for $X \subset \reals^2$ with two points.}
\label{f5}
\end{center}
\end{figure}

 Now let $X \subset \reals^2$ be an arbitrary $D_2$-symmetric set and let $O$ be the barycenter of $X$. Then $X \setminus \{O\}$ is the union of $b$ 2-gons and $c$ rectangles. A 2-gon consists of a pair of points equidistant from $O$. Two 2-gons in $X$ must be either homothetic or rotated relatively to one another by $\frac{\pi}{2}$. In the case $O \notin X$ and $|X|=N=4$, $X$ may correspond to a collinear configuration of four points, a rhombus or a rectangle. We show that the simplest possible block-diagonal form for $A+\kappa J$ is achieved in each of these cases.
 
 Before beginning our analysis, which will be succint, let us notice that the $D_2$-symmetric sets formed by four collinear points or the vertices of a rhombus consist of the union of two $D_2$-orbits, while the rectangle corresponds to a single $D_2$-orbit. As usual, we assume the masses in each orbit are equal. The following lemma contains a helpful remark.
 
\begin{lemma} \label{redstan} The standard representation $\sigma$ of $D_2$ is not irreducible. More precisely, we have that $\sigma \simeq \phi \oplus \psi$.
\end{lemma}
 
\subsection{Four collinear points and rhombus} We can directly apply the results of subsections~\ref{dilationrotation} and~\ref{transl}, keeping lemma~\ref{redstan} in mind. Since $\sections$ is eight-dimensional, each of the four two-dimensional isotypic components contains an eigenvector of $A$, and since the bases constructed in~\ref{dilationrotation} and~\ref{transl} are $M$-orthogonal, we have that such bases are formed by displacements corresponding to eigenvectors of $A$. Thus, with respect to the bases constructed in subsections~\ref{dilationrotation} and~\ref{transl}, $A$ is diagonalized and $J$ is in standard form. We conclude that $A+\kappa J$ assumes a block-diagonal form with $2 \times 2$ blocks on the diagonal.

\subsection{Rectangle} All the remarks and conclusions in the previous paragraph apply to rectangles, with the caveat that we must refer to subsection~\ref{sec8} and the translation decomposition described in~\ref{transdecII} for the construction of the appropriate bases. So also for rectangles the operators $A+\kappa J$ assume the simplest block-diagonal form\footnote{A square can be seen alternatively as a $D_2$ or $D_4$-symmetric set, but this seems not to significantly affect the final block diagonalization form of the linearization matrix.}.


\begin{thebibliography}{99}
\bibitem{a91} M. Artin, Algebra, Prentice-Hall, New Jersey (1991).
\bibitem{fs94} A. F\"{a}ssler, E. Stiefel, Group Theoretical Methods and Their Applications, Birkh\"{a}user, Boston (1992).
\bibitem{mw93} J. Mehra, A. Wightman, B. Judd, G. Mackey, eds., The Collected Works of Eugene Paul Wigner, Part A, The Scientific Papers, Vol. 1, Part III, The Mathematical Papers Annotated by George Mackey, Springer-Verlag, Berlin Heidelberg (1993)
\bibitem{ma56} J. C. Maxwell, On the Stability of the Motion of Saturn's Rings. In W. D. Niven, ed., The Scientific Papers of James Clerk Maxwell, Vol. 1, Dover, New York (1965).
\bibitem{m95} R. Moeckel, Linear Stability Analysis of Some Symmetrical Classes of Relative Equilibria, IMA, vol. 63, Springer, New York (1995).
\bibitem{p76} J. Palmore, Measure of degenerate relative equilibria. I, Ann. of Math. {\bf{104}} (1976) pp. 421-429.
\bibitem{p03} H. Poincar\'e, Figures d'\'equilibre d'une masse fluide; Le\c cons profess\'ees \`a la Sorbonne en 1900, Gauthier-Villars, Paris (1902).
\bibitem{gr13} G. Roberts, Stability of Relative Equilibria in the $N$ Vortex Problem, SIAM J. Appl. Dyn. Syst., {\bf 12}, no. 2 (2013), pp. 1114-1134.
\bibitem{sv91} D. Scheeres, N. Vinh, Linear Stability of a Self-Gravitating Ring, Celestial Mech. Dynam. Astronom., {\bf{51}} (1991) pp. 83-103.
\bibitem{jps77} J.-P. Serre, Linear Representations of Finite Groups, Springer, New York (1977).
\bibitem{s94} S. Sternberg, Group Theory and Physics, Cambridge Univ. Press, Cambridge (1994).
\end{thebibliography}
\end{document}